\documentclass{article}
\usepackage{tikz}
\usetikzlibrary{arrows.meta,positioning,calc,fit,backgrounds,shapes.geometric,decorations.pathreplacing}
\usepackage{algorithm}
\usepackage{algpseudocode}
\usepackage{arxiv}
\usepackage{enumitem}
\usepackage[utf8]{inputenc} 
\usepackage[T1]{fontenc}    
\usepackage{hyperref}       
\usepackage{url}            
\usepackage{booktabs}       
\usepackage{amsfonts}       
\usepackage{nicefrac}       
\usepackage{microtype}      
\usepackage{lipsum}
\usepackage{fancyhdr}       
\usepackage{graphicx}       
\usepackage{amsmath}
\usepackage{lyu}
\usepackage{amsthm}
\usepackage{amssymb}
\usepackage{xcolor}
\usepackage{subcaption}
\usepackage{bbm}
\newtheorem{theorem}{Theorem}[section]

\newtheorem{lemma}[theorem]{Lemma}
\theoremstyle{definition}

\newtheorem{assumption}{Assumption}
\newtheorem{proposition}[theorem]{Proposition}
\newtheorem{remark}[theorem]{Remark}
\graphicspath{{}{../}{./}}
\title{
Multiscale Nudging: From Macroscopic Observations to Microscopic Dynamics
}

\author{
  Liyao Lyu\\
  Department of Mathematics \\
  University of California, Los Angeles \\
  Los Angeles, CA 90024, USA\\
  \texttt{lyuliyao@math.ucla.edu} \\
  \AND
   Xinyue Yu \\
  Department of Mathematics \\
  University of California, Los Angeles \\
  Los Angeles, CA 90024, USA\\
  \texttt{tracy@math.ucla.edu} \\
  \AND
   Hayden Schaeffer  \\
  Department of Mathematics \\
  University of California, Los Angeles \\
  Los Angeles, CA 90024, USA\\
  \texttt{hayden@math.ucla.edu} \\
}

\begin{document}
\maketitle

\begin{abstract}
    We introduce a measure-based nudging framework for assimilating macroscopic observations into microscopic mean-field particle dynamics. The central difficulty is a representation mismatch: the forecast is a labeled particle system, while the observations specify only a smoothed, permutation-invariant density. To address this mismatch, we define the forecast-observation discrepancy as a quadratic functional on probability measures after applying the same smoothing operator used by the observation process. The Wasserstein gradient of this functional induces a transport velocity on state space, which yields a particle-level correction without constructing particle-to-particle matching, linearizing the dynamics, or estimating ensemble covariances. For a fixed observation scale, we prove well-posedness of the assimilated McKean--Vlasov dynamics and propagation of chaos for the interacting particle approximation. Under exact smoothed observations and an observability condition at the kernel scale, we establish an \(L^2\)-stability estimate showing exponential decay up to a bias floor controlled by model misspecification. Numerical experiments on linear, bimodal, chaotic, kinetic, and collective-motion systems demonstrate that the method can recover macroscopic structure from incomplete density-level observations.
\end{abstract}


\maketitle

\section{Introduction}
Mean-field particle systems are widely used tools in modeling many natural and engineered systems with multiscale phenomena, which arise in fluid dynamics, neuroscience, materials science, and biological systems~\cite{jost2005dynamical,wiggins2003introduction}. These multiscale systems typically involve a large number of degrees of freedom, making direct analysis and simulation at the microscopic level prohibitively expensive. Introduced by ~\cite{kac1956foundations,mckean1966class}, mean-field theory has provided a way to connect microscopic interactions with macroscopic dynamics. The mean-field limit replaces the high-dimensional coupled dynamics of $N$ interacting agents with a single representative equation whose drift depends on the law of the process itself, a formulation known as the McKean-Vlasov stochastic differential equation. On the theoretical side, substantial progress has been made on well-posedness~\cite{de2020strong,huang2019distribution}, existence and uniqueness of solutions~\cite{mishura2020existence}, and propagation of chaos~\cite{sznitman2006topics,chaintron2021propagation,jabin2017mean}. On the applied side, mean-field models have found successful applications in areas ranging from mean-field control~\cite{buckdahn2017mean} and mean-field games~\cite{ding2022mean} to high-dimensional sampling~\cite{liu2017stein,carrillo2024fisher} and neural network training~\cite{rotskoff2022trainability}.

Although mean-field modeling has achieved broad success, several practical challenges remain. In many complex systems, the true mean-field dynamics are unknown, and thus a learned or approximated drift is needed. One approach is to approximate the governing system ~\cite{lu2019nonparametric,lu2021learning,lyu2026mvnn,chen2021data,du2022discovery,sharrock2021parameter,zhang2024bayesian, schaeffer2020extracting,liu2023random,lu2024learning, schaeffer2017learning, schaeffer2017sparse}, which can often introduce bias during forecasting. This issue is often observed in strongly nonlinear, heterogeneous, or transitional regimes, which are difficult to capture using simplified effective dynamics alone~\cite{sznitman2006topics,jabin2018quantitative,cattiaux2008probabilistic}. Even when the dominant mechanisms are reasonably well understood, uncertainty in initialization and unresolved interactions can accumulate over time, producing significant forecast drift. 

A natural way to reduce such forecast drift is data assimilation~\cite{lee2015multiscale,deng2025lemda,harlim2013test,deng2024particle,burov2021kernel}, which uses observational information to constrain the forecast and correct departures from the reference dynamics. In the mean-field problems considered here, however, the forecast and the observations are represented at different physical scales. The model evolves microscopic particles or agents, such as atoms, molecules, cells, or individual animals, whereas the available data often provide only macroscopic information. In image-based settings, for example, pixel intensities or spatial averages describe a coarse-grained density field rather than the positions and identities of individual particles~\cite{carrillo2021mean,paul2022microscopic}. Direct access to individual particle positions is therefore unavailable. This creates a cross-scale assimilation problem. The forecast contains a finite particle configuration, but the observation specifies only how mass is distributed at the resolution of the measurement device. Hence, the data do not determine a unique microscopic configuration, nor do they provide a canonical correspondence between observed mass and simulated particles. Even when particle locations can be extracted from measurements, the resulting tracks are often unlabeled, partially observed, or inconsistent across time. The assimilation procedure must therefore compare forecast and data at the observational scale, rather than through direct particle-to-particle matching.

Classical data assimilation has been a powerful tool, particularly in weather prediction and geophysical forecasting. The Ensemble Kalman Filter (EnKF)~\cite{evensen2003ensemble,evensen2009data} and its many variants approximate the Bayesian filtering distribution by propagating an ensemble of model states and updating them via a Kalman-type correction. In the present setting, however, the relevant observation is a permutation-invariant functional of the empirical distribution, whereas the standard EnKF update is built from covariances between labeled state coordinates. This can introduce an artificial dependence on particle indexing unless one imposes additional structure, such as reliable particle identities, a canonical ordering, or an explicit matching rule.  Variational methods such as 3D-Var and 4D-Var~\cite{lorenc1986analysis,courtier1994strategy} can incorporate nonlinear observation operators, but using coarse density observations would require optimizing over microscopic particle configurations whose macroscopic density matches the data. This lifting from density space back to particle space is highly nonunique. Particle filters~\cite{chopin2020introduction} provide a fully nonlinear Bayesian alternative, but they suffer from weight degeneracy in high-dimensional systems and do not by themselves resolve the permutation-invariant nature of density observations.  More recently, feedback-based approaches, commonly referred to as nudging or synchronization methods~\cite{azouani2014continuous,bessaih2015continuous}, have attracted much attention due to their simplicity and suitability to rigorous analysis~\cite{newey2025model,lu2025continuous,clark2018inferring,farhat2020data}. In nudging, a relaxation term continuously drives the model state toward consistency with observations, and convergence can often be established under verifiable spectral-gap or dissipativity conditions~\cite{albanez2016continuous,jolly2017data}.
Between these classical filters and nudging lies a family of transport- and coupling-based mean-field filters that move particles by a deterministic feedback velocity ~\cite{yang2011feedback,gregory2016multilevel,daum2010exact,pulido2019sequential,calvello2025ensemble,bao2024ensemble,bao2024score}, each of which assumes a likelihood or observation operator acting on the labeled state, rather than on the permutation-invariant density considered here.
However, in the present setting, where the state is represented by a finite ensemble of particles while the observations correspond to coarse-grained density fields, a fundamental issue remains: the discrepancy between model and data is naturally defined at the level of measures, rather than labeled states. This mismatch makes a standard $L^2$ residual unsuitable for comparing particle-based forecasts with coarse-grained density observations and may result in ill-posed or numerically unstable formulations.

This motivates a formulation of the assimilation problem at the level of measures. Rather than constructing a correction in a labeled \(N\)-particle state space, we define the forecast-observation discrepancy as a functional of the forecast law. More precisely, we regularize the empirical forecast measure using the same smoothing operator that defines the observed density, and then evaluate the discrepancy between the resulting coarse-grained forecast density and the data. Thus the feedback is driven by quantities that are invariant under permutations of the forecast particles and defined at the same resolution as the measurements. This avoids choosing artificial particle correspondences or lifting a coarse density observation back to a unique microscopic configuration. The central object is therefore an observation-scale, kernel-regularized misfit functional on probability measures.

Given this measure-level misfit, we seek a mechanism for translating the resulting correction to the particle level. We address this through the Wasserstein gradient-flow method ~\cite{jordan1998variational, ambrosio2005gradient,chen2023gradient,chen2023sampling}. This theory provides a variational formulation for a broad class of evolution equations including the Fokker--Planck equation, porous medium equations, and aggregation-diffusion models~\cite{carrillo2003kinetic,carrillo2019aggregation}. Within this method, the Wasserstein gradient of a regularized observation-misfit functional defines a transport velocity field on state space.  Evaluating this velocity field at the forecast particles yields a microscopic correction that is consistent with the macroscopic density-level discrepancy.  The resulting scheme compares forecast and data at the macroscopic density scale, yet implements the correction through microscopic particle dynamics without requiring particle labels or pointwise matching.

In this paper, we develop Multiscale Nudging as a measure-level method for assimilating coarse observations into microscopic mean-field dynamics. The forecast-observation mismatch is defined on probability measures after applying the same smoothing operator used in the observations, so the correction is invariant under particle relabeling and does not require lifting a coarse density to a unique particle configuration. We derive the feedback term as the Wasserstein steepest-descent direction of this misfit and evaluate the resulting velocity at particle locations. For fixed bandwidth and nudging strength, we prove well-posedness of the assimilated McKean-Vlasov dynamics and propagation of chaos for the particle system. With exact smoothed observations and a kernel-scale observability condition, we also obtain an $L^2$-error estimate with exponential decay up to a model-error-dependent bias floor. The method is tested on Gaussian, multimodal, chaotic, kinetic, and collective-motion examples. Figure~\ref{fig:schematic} summarizes the overall pipeline: the empirical forecast is coarse-grained and compared with the observation at the macroscopic scale, while the resulting correction is applied to individual particles at the microscopic scale through the Wasserstein gradient.
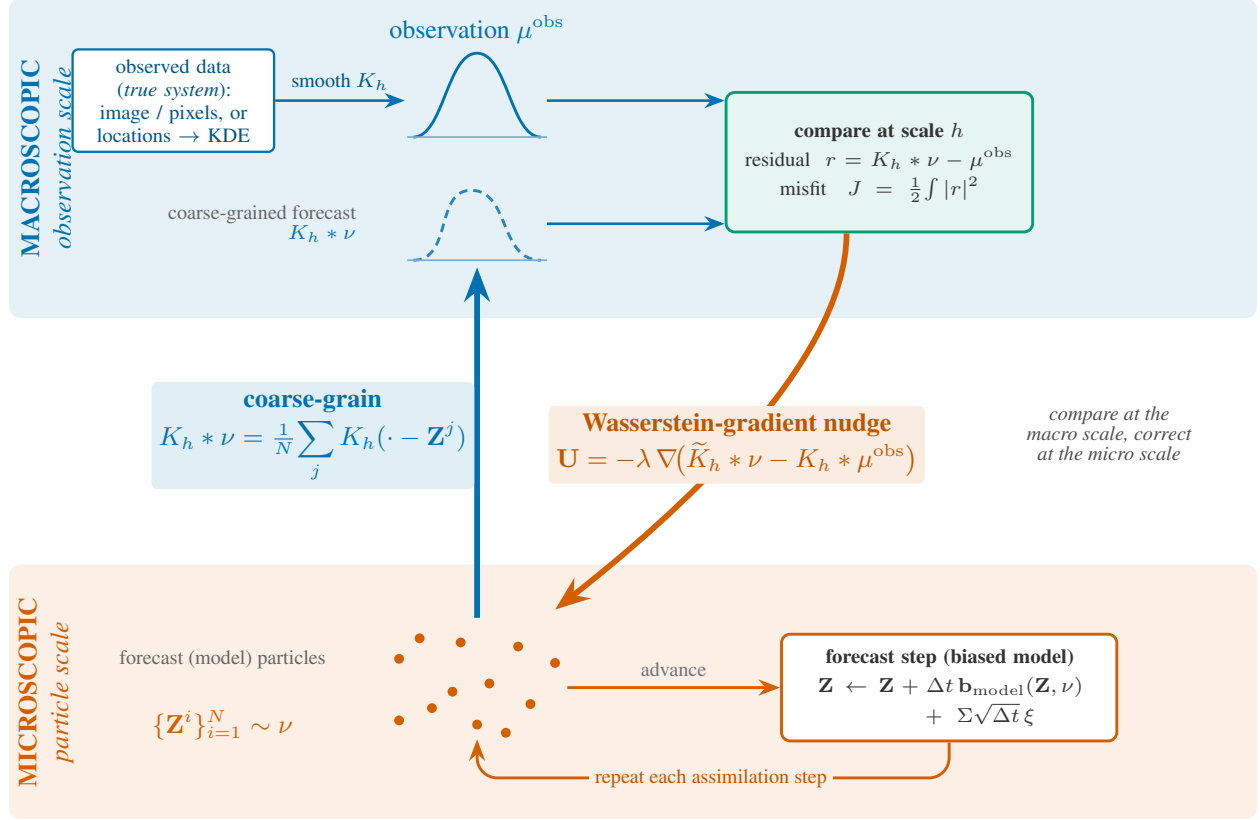
\begin{figure}[t]
\centering
\definecolor{macroC}{HTML}{0072B2}   
\definecolor{microC}{HTML}{D55E00}   
\definecolor{accentC}{HTML}{009E73}  
\resizebox{\linewidth}{!}{%
\begin{tikzpicture}[
  >=Stealth,
  font=\small,
]

\def\laneL{0}      
\def\laneR{15.2}   
\def\topB{6.1}     
\def\topT{10.0}    
\def\botB{0.0}     
\def\botT{3.1}     

\begin{scope}[on background layer]
  \fill[macroC!10,rounded corners=4pt] (\laneL,\topB) rectangle (\laneR,\topT);
  \fill[microC!10,rounded corners=4pt] (\laneL,\botB) rectangle (\laneR,\botT);
\end{scope}

\node[rotate=90, anchor=center, text=macroC, font=\bfseries\footnotesize, align=center]
  at (0.45,{(\topB+\topT)/2}) {MACROSCOPIC\\[1pt]\itshape\mdseries observation scale};
\node[rotate=90, anchor=center, text=microC, font=\bfseries\footnotesize, align=center]
  at (0.45,{(\botB+\botT)/2}) {MICROSCOPIC\\[1pt]\itshape\mdseries particle scale};


\node[draw=macroC, fill=white, rounded corners=2pt, line width=0.8pt,
      align=center, font=\scriptsize, text=macroC!75!black,
      minimum height=0.95cm, text width=2.22cm]
  (src) at (2.0,8.75) {observed data\\(\emph{true system}):\\image / pixels, or\\locations $\to$ KDE};

\coordinate (obsc) at (5.7,8.75);                      
\begin{scope}[shift={(obsc)},scale=0.74]
  \draw[macroC,line width=1pt] (-1.05,-0.6)
    .. controls (-0.55,-0.6) and (-0.5,0.78) .. (0,0.78)
    .. controls (0.5,0.78) and (0.55,-0.6) .. (1.05,-0.6);
  \draw[macroC!55,line width=0.6pt] (-1.15,-0.6) -- (1.15,-0.6);
\end{scope}
\node[text=macroC, font=\footnotesize] at (5.7,9.62) {observation $\mu^{\mathrm{obs}}$};
\draw[->,macroC,line width=0.9pt] (src.east) -- (4.82,8.75)
  node[midway, above, font=\scriptsize, text=macroC!75!black] {smooth $K_h$};

\coordinate (fcc) at (5.7,7.25);                       
\begin{scope}[shift={(fcc)},scale=0.74]
  \draw[macroC!85,line width=1pt,densely dashed] (-1.05,-0.6)
    .. controls (-0.4,-0.6) and (-0.7,0.55) .. (-0.1,0.55)
    .. controls (0.55,0.55) and (0.35,-0.6) .. (1.05,-0.6);
  \draw[macroC!55,line width=0.6pt] (-1.15,-0.6) -- (1.15,-0.6);
\end{scope}
\node[font=\scriptsize, text=black!60, align=right, anchor=east] at (4.35,7.25)
  {coarse-grained forecast\\[-1pt]\textcolor{macroC}{$K_h*\nu$}};

\node[draw=accentC, fill=accentC!10, rounded corners=3pt, line width=1pt,
      align=center, font=\scriptsize, text width=3.5cm, minimum height=1.7cm,
      text=black!80]
  (mis) at (10.6,8.0)
  {\textbf{compare at scale }$h$\\[2pt]
   residual $\;r = K_h*\nu - \mu^{\mathrm{obs}}$\\[2pt]
   misfit $\;J = \tfrac{1}{2}\!\int |r|^2$};
\draw[->,macroC,line width=0.9pt] (6.55,8.75) -- ([yshift=0.75cm]mis.west);  
\draw[->,macroC,line width=0.9pt] (6.55,7.25) -- ([yshift=-0.75cm]mis.west); 


\coordinate (cloud) at (5.7,1.6);
\begin{scope}[shift={(cloud)}]
  \foreach \p in {(-0.95,0.35),(-0.55,-0.25),(-0.2,0.55),(0.15,0.05),
                  (0.5,0.5),(0.0,-0.45),(0.65,-0.2),(-0.7,0.6),
                  (0.95,0.3),(-0.3,-0.05),(0.35,-0.55),(-0.95,-0.4)}{
    \fill[microC] \p circle (1.7pt);
  }
\end{scope}
\node[font=\scriptsize, text=black!60] at (2.6,1.95) {forecast (model) particles};
\node[text=microC, font=\footnotesize] at (2.6,1.15) {$\{\mathbf{Z}^i\}_{i=1}^{N}\sim\nu$};

\node[draw=microC, fill=white, rounded corners=3pt, line width=1pt,
      align=center, font=\scriptsize, text width=3.85cm, minimum height=1.3cm,
      text=black!80]
  (adv) at (11.45,1.6)
  {\textbf{forecast step (biased model)}\\[2pt]
   $\mathbf{Z} \leftarrow \mathbf{Z} + \Delta t\, \mathbf{b}_{\mathrm{model}}(\mathbf{Z},\nu)$\\[1pt]
   $\hphantom{\mathbf{Z} \leftarrow{}}+\,\Sigma\sqrt{\Delta t}\,\xi$};

\draw[->,microC,line width=1pt] (6.8,1.6) -- (adv.west)
  node[midway, above, font=\scriptsize, text=black!55] {advance};

\draw[->,microC,line width=1pt,rounded corners=7pt]
  (adv.south) -- (11.45,0.5) -- (5.7,0.5) -- (5.7,0.97);
\node[font=\scriptsize, text=microC!90!black, fill=microC!10, inner sep=2pt]
  at (8.55,0.5) {repeat each assimilation step};


\draw[->,macroC,line width=2pt] (5.7,2.45) -- (5.7,6.72)
  node[pos=0.52, anchor=east, fill=macroC!12, rounded corners=2pt, inner sep=2.5pt,
       text=macroC, font=\footnotesize, align=center]
  {\textbf{coarse-grain}\\[1pt]$\displaystyle K_h*\nu=\tfrac1N\!\sum_j K_h(\cdot-\mathbf{Z}^j)$\hspace{2pt}};

\draw[->,microC,line width=2pt]
  ($(mis.south)+(-0.4,0)$) .. controls +(0,-1.6) and +(1.4,1.1) ..
  ($(cloud)+(0.95,0.95)$)
  node[pos=0.52, fill=microC!12, rounded corners=2pt, inner sep=2.5pt,
       text=microC!90!black, font=\footnotesize, align=center]
  {\textbf{Wasserstein-gradient nudge}\\[1pt]
   $\mathbf{U}=-\lambda\,\nabla\!\big(\widetilde{K}_h*\nu - K_h*\mu^{\mathrm{obs}}\big)$};

\node[text=black!70, font=\scriptsize\itshape, align=center, anchor=center]
  at (13.4,4.7) {compare at the\\[-1pt]macro scale, correct\\[-1pt]at the micro scale};

\end{tikzpicture}%
}
\caption{\textbf{The Multiscale Nudging pipeline.} Biased-model forecast particles
$\{\mathbf{Z}^i\}\sim\nu$ (\emph{bottom band, microscopic}) are coarse-grained
with the kernel $K_h$ and compared with the observed density $\mu^{\mathrm{obs}}$
(\emph{top band, macroscopic}). The Wasserstein gradient of the resulting misfit
then nudges every particle back at the microscopic scale, and the step repeats.
The comparison lives at the observation scale $h$, but the correction acts on
individual particles---without matching, linearization, or ensemble covariances.}
\label{fig:schematic}
\end{figure}

\section{Method}
Let $(\Omega, \mathcal{F}, \mathbb{P})$ be a probability space with a filtration $(\mathcal{F}_t)_{t \ge 0}$. We consider the reference, or ground-truth, dynamics governed by the following McKean-Vlasov stochastic differential equation: 
\begin{equation}\label{equ:true_mean_field}
\intd \bX_t  = \bb_{\mathrm{true}} (\bX_t,\mu_t)\intd t + \Sigma\intd \bW_t, \quad \bX_0 \sim \mu_0,
\end{equation}
where $\mu_t= \mathrm{Law}(\bX_t)\in \mathcal P_2(\mathbb R^d)$ denotes the probability distribution of the true process, $\bX_t\in \mathbb R^d$ is the state of the reference particle, $\bW_t$ is a standard $d$-dimensional Brownian motion, $\bb_{\mathrm{true}}:\mathbb R^d\times \mathcal P_2(\mathbb R^d)\to \mathbb R^d$ is the true drift term, and $\Sigma  \in \mathbb R^{d\times d}$ is the diffusion coefficient (assumed constant for simplicity). The evolution of $\mu_t$ is given by 
\begin{equation}\label{equ:true_density}
\partial_t \mu_t
= -\nabla_{\mathbf{x}} \cdot \big( \mathbf{b}_{\mathrm{true}}(\mathbf{x},\mu_t)\,\mu_t \big)+\frac{1}{2}\nabla_{\mathbf{x}} \cdot \big( \Sigma \Sigma^\top \nabla_{\mathbf{x}} \mu_t \big).
\end{equation}
Equation \eqref{equ:true_mean_field} can be viewed as the mean-field approximation of an interacting particle system 
\[
\intd \bX_t^i  = \bb(\bX_t^1,\cdots, \bX_t^N)\intd t  + \Sigma \intd \bW_t^i, \qquad i=1,\ldots,N,
\]
with $N\to \infty$, under appropriate assumptions ensuring propagation of chaos \cite{chaintron2022propagation,chaintron2021propagation}. In practice, the true drift $\bb_{\mathrm{true}}$ is typically unknown. A considerable body of recent work has focused on learning interaction laws from data~\cite{lu2019nonparametric,lu2021learning,lyu2026mvnn}. The learned interaction model induces an approximate drift $\bb_{\mathrm{model}}$, which generally differs from the true drift with residual
\[
\bR(\bx,\mu) := \bb_{\mathrm{model}}(\bx,\mu) -  \bb_{\mathrm{true}}(\bx,\mu).
\]
The corresponding approximate dynamics are given by
\[
\intd \bY_t = \bb_{\mathrm{model}}(\bY_t,\mu^{\mathrm{model}}_t)\intd t + \Sigma \intd \bW_t,
\]
where $\mu^{\mathrm{model}}_t = \mathrm{Law}(\bY_t)$, or at the particle level,
\begin{equation}\label{equ:model}
\intd \bY^{i,N}_t = \bb_{\mathrm{model}}(\bY^{i,N}_t,\mu^{\mathrm{model},N}_t)\intd t + \Sigma \intd \bW^{i,N}_t,
\end{equation}
where $\mu^{\mathrm{model},N} = \frac{1}{N}\sum_{i=1}^N \delta_{\bY_i}$. 
Due to initialization error and model misspecification, the discrepancy between the true and approximate systems tends to accumulate over time. 

Data assimilation is one strategy to reduce this drift by continuously incorporating observational information into the model evolution. 
However, in the mean-field setting, one difficulty is that the representative particle $\bX_t$ is not itself a physical observable, and even in the interacting particle system $\{\bX_t^i\}_{i=1}^N$, individual trajectories are rarely accessible. In many practical applications, the available measurements take the form of aggregate or image-based data, such as density fields reconstructed from pixel observations, from which one cannot establish a correspondence between observed positions and particle indices. For example, in collective motion experiments such as \cite{ballerini2008interaction, cavagna2010scale, zhang2010collective}, the accessible quantity is a coarse-grained spatial density reconstructed from imaging data, rather than the full microscopic state of each individual. In geophysical forecasting \cite{evensen2003ensemble}, satellite observations similarly provide smoothed integrals of the atmospheric state distribution. In all these settings, the observational quantity takes the form of a smoothed density. For simplicity, we the following two forms of smoothed observations in this paper. Let $K_h (\bx) = h^{-d}K(\bx/h)$ be a smoothing kernel with bandwidth $h>0$, where $K:\mathbb R^d \to \mathbb R_+$. Concretely, we consider two cases. In the first case, one has access to pixel-level image data, from which the observed density is directly reconstructed through a coarse-graining operator
\[
\mu_t^{\mathrm{obs}}(\bz) = \mathcal{O}_h(\mu_t)(\bz) = \int K_h(\bz - \bx)\,\mu_t(\mathrm{d}\bx).
\]
In the second case, only a partial and unlabeled set of particle locations $\{\bX_t^{\mathrm{obs},k,M}\}_{k=1}^M$ is available, without consistent index correspondence across time, and the observed density is approximated by the kernel density estimator
\[
\mu_t^{\mathrm{obs}}(\bz) = \frac{1}{M}\sum_{k=1}^{M} K_h(\bz - \bX_t^{\mathrm{obs},k,M}).
\]
To reduce discrepancies caused by imperfect initialization and model error, we propose to augment the approximate mean-field dynamics with a nudging term, in the spirit of continuous data assimilation \cite{bessaih2015continuous}. This term drives the evolving law $\nu_t$ toward consistency with the available observations.

Assume that the smoothing kernel is even, i.e. \(K_h(\bx)=K_h(-\bx)\). 
For a forecast law \(\nu\) and a fixed observation \(\mu^{\mathrm{obs}}\), define the observation-scale residual
\[
r_{\nu,h} := K_h*\nu-\mu^{\mathrm{obs}} .
\]
We measure the forecast--observation mismatch by
\[
J_{\mathrm{nud},h}(\nu,\mu^{\mathrm{obs}})
=
\frac12\int_{\mathbb R^d}
\left|r_{\nu,h}(\bx)\right|^2\,\intd\bx .
\]
Since convolution with an even kernel is self-adjoint in \(L^2\), differentiating \(J_{\mathrm{nud},h}\) along signed measure perturbations gives the first variation
\[
\phi_{\nu,h}(\bz)
:=
\frac{\delta J_{\mathrm{nud},h}}{\delta \nu}(\bz)
=
(K_h*r_{\nu,h})(\bz)
=
\bigl(\widetilde K_h*\nu-K_h*\mu^{\mathrm{obs}}\bigr)(\bz),
\qquad
\widetilde K_h:=K_h*K_h .
\]
We define the nudging correction as the Wasserstein steepest-descent direction of \(J_{\mathrm{nud},h}\), with \(\mu^{\mathrm{obs}}\) held fixed during the correction step. Equivalently, for an artificial nudging time step \(\tau>0\), the corrected law is formally given by the minimizing-movement problem
\[
\nu^{+}
\in
\underset{\rho\in\mathcal P_2(\mathbb R^d)}{\operatorname{argmin}}
\left\{
\frac{1}{2\tau}W_2^2(\rho,\nu)
+
\lambda J_{\mathrm{nud},h}(\rho,\mu^{\mathrm{obs}})
\right\}.
\]
In infinitesimal form, the corresponding tangent-space problem is
\[
\bu_{\mathrm{nud},h}
=
\underset{\bu\in T_\nu\mathcal P_2}{\operatorname{argmin}}
\left\{
\int_{\mathbb R^d}
\nabla\phi_{\nu,h}(\bz)\cdot \bu(\bz)\,\intd\nu(\bz)
+
\frac{1}{2\lambda}
\int_{\mathbb R^d}
|\bu(\bz)|^2\,\intd\nu(\bz)
\right\}.
\]
The Euler--Lagrange condition for this quadratic problem yields
\[
\bu_{\mathrm{nud},h}(\bz,\nu,\mu^{\mathrm{obs}})
=
-\lambda\nabla\phi_{\nu,h}(\bz).
\]
Therefore,
\[
\bu_{\mathrm{nud},h}(\bz,\nu,\mu^{\mathrm{obs}})
=
-\lambda\nabla
\bigl(
\widetilde K_h*\nu
-
K_h*\mu^{\mathrm{obs}}
\bigr)(\bz).
\]
Substituting into~\eqref{equ:model}, we obtain the regularized assimilated dynamics
\begin{equation}\label{equ:assimilated_regularized}
\intd\bZ_t
= \bb_{\mathrm{model}}(\bZ_t,\nu_t)\,\intd t -\lambda  \nabla \left( \tilde K_h * \nu - K_h * \mu^{\mathrm{obs}} \right)(\bZ_t)\,\intd t
+ \Sigma \intd \bW_t.
\end{equation}
Formally, the Law of $\bZ_t$, denoted by $\nu_t$, satisfies the following nonlinear Fokker-Planck equation:
\begin{equation}\label{equ:fpk_regularized}
\partial_t \nu_t
=
-\nabla\cdot(\nu_t b_{\rm model}(\cdot,\nu_t))
+
\nabla\cdot(A\nabla\nu_t)
+
\lambda\nabla\cdot
\left[
\nu_t\nabla(K_h*r_t)
\right],
\qquad
A = \frac12\Sigma\Sigma^\top.
\end{equation}
\begin{remark}
    If we assume that $\bb_{\mathrm{model}}$ and $\Sigma$ take the following form $$ \bb_{\mathrm{model}}(\bz,\nu) =  -\nabla \left( \frac{\delta F_{\mathrm{model}}}{\delta \nu}\right)(\bz), \quad \Sigma = \sqrt{2\beta} \mathbf{I}, $$ for some functional $F_{\model}(\mu)$, then the assimilated dynamics \eqref{equ:assimilated_regularized} can be interpreted as a Wasserstein gradient flow of the free energy 
\begin{equation}\label{equ:free_energy}
\mathcal F_t (\nu) = F_{\mathrm{model}}(\nu) + \lambda J_{\mathrm{nud},h}(\nu,\mu_t^{\mathrm{obs}}) + \beta \mathrm{Ent} (\nu),
\end{equation}
where $\mathrm{Ent} (\nu) = \int \nu\log(\nu)\intd \bz$.
\end{remark}

\begin{algorithm}[H]
\caption{Multiscale Nudging: particle implementation from coarse density observations}
\label{alg:multiscale-nudge}
\begin{algorithmic}[1]
\raggedright
\Require Particles $\{\bZ_0^i\}_{i=1}^N$, model drift $\bbm$, diffusion $\Sigma$, step $\Delta t$.
\Statex \hphantom{\textbf{Input:} }Grid $\mathcal G=\{\bx_q\}_{q=1}^G$, weights $\{w_q\}$, kernel $K_h$, obs.\ times $\mathcal T_{\mathrm{obs}}$.
\Statex \hphantom{\textbf{Input:} }Nudging strength $\lambda$, substeps $L_{\mathrm{nud}}$, substep size $\Delta\tau=\Delta t/L_{\mathrm{nud}}$.
\Ensure Assimilated trajectories $\{\bZ_n^i\}_{i=1}^N$.
\Statex\hrulefill
\For{$n=0,\ldots,N_t-1$}
  \State Empirical forecast law $\nu_n^N=\frac1N\sum_{i=1}^N\delta_{\bZ_n^i}$.
  \State \textbf{Forecast.} Push each particle with $\boldsymbol\xi_n^i\sim\mathcal N(0,I_d)$:
  \algeq{\widetilde{\bZ}_{n+1}^i=\bZ_n^i+\Delta t\,\bbm(\bZ_n^i,\nu_n^N)+\Sigma\sqrt{\Delta t}\,\boldsymbol\xi_n^i.}
  \State \textbf{Observe.} Receive coarse density $y_{n+1,q}\approx\mub_{t_{n+1}}(\bx_q)$ on $\mathcal G$.
  \State Initialize $\bZ_{n+1}^{i,(0)}\gets\widetilde{\bZ}_{n+1}^i$.
  \For{$\ell=0,\ldots,L_{\mathrm{nud}}-1$}
    \State Smoothed density $\rho_q^{(\ell)}=\frac1N\sum_{j=1}^N K_h(\bx_q-\bZ_{n+1}^{j,(\ell)})$ for all $q$.
    \State Residual $r_q^{(\ell)}=\rho_q^{(\ell)}-y_{n+1,q}$.
    \State \textbf{Nudge.} Wasserstein-gradient velocity at each particle:
    \algeq{\bU_i^{(\ell)}=-\lambda\sum_{q=1}^G w_q\,\nabla_{\bz}K_h(\bz-\bx_q)\big|_{\bz=\bZ_{n+1}^{i,(\ell)}}\,r_q^{(\ell)}.}
    \State Substep $\bZ_{n+1}^{i,(\ell+1)}\gets\bZ_{n+1}^{i,(\ell)}+\Delta\tau\,\bU_i^{(\ell)}$.
  \EndFor
  \State Commit $\bZ_{n+1}^i\gets\bZ_{n+1}^{i,(L_{\mathrm{nud}})}$.
\EndFor
\end{algorithmic}
\end{algorithm}
\section{Theoretical Analysis}
This section establishes three properties of the Multiscale Nudging scheme: the
kernel-regularized feedback yields well-posed dynamics, the finite-particle
implementation converges to the mean-field model, and the feedback provably
reduces the forecast error. The analysis begins with Lemma~\ref{lemma:kernel}, which shows the two key properties of the smoothed
kernel $\tilde K_h$: a globally Lipschitz gradient, supplying the regularity for
well-posedness, and the $H^1$ approximation $\tilde K_h * v \to v$ as
$h\downarrow 0$, quantifying the accuracy of observing only at scale $h$.
Proposition~\ref{prop:well_possedness} then gives a unique strong solution to the
assimilated McKean-Vlasov dynamics~\eqref{equ:assimilated_regularized}, and
Proposition~\ref{equ:propagation_of_chaos} establishes propagation of chaos,
justifying the particle discretization of Algorithm~\ref{alg:multiscale-nudge}.
With the strict positivity from Proposition~\ref{prop:positivity}, our main
result, Theorem~\ref{thm:main_kdenudge}, shows that under a kernel-scale
observability condition and a lower bound on the nudging strength, the $L^2$
error decays exponentially to a floor of order $\Delta^2$, vanishing when the
model is exact.
\begin{lemma}[Periodized Gaussian kernel]\label{lemma:kernel}
Let \(\Omega=\mathbb T^d=\mathbb R^d/\mathbb Z^d\). For \(h>0\), define the periodized Gaussian kernel
\[
K_h^{\mathbb T}(\bx)
=
\sum_{\bm m\in\mathbb Z^d}
h^{-d}\pi^{-d/2}
\exp\left(-\frac{\|\bx+\bm m\|^2}{h^2}\right),
\qquad \bx\in\mathbb T^d .
\]
Let
\[
\widetilde K_h
=
K_h^{\mathbb T}*_{\mathbb T}K_h^{\mathbb T},
\]
where \(*_{\mathbb T}\) denotes convolution on the torus. Then:
\begin{enumerate}
    \item \(K_h^{ \mathbb T}\) and \(\widetilde K_h\) belong to
    \(C^\infty(\mathbb T^d)\). In particular, for each fixed \(h>0\),
    \(\widetilde K_h\) is globally Lipschitz on \(\mathbb T^d\).

    \item For each fixed \(h>0\),
    \[
    \nabla\widetilde K_h\in W^{1,\infty}(\mathbb T^d).
    \]
    In particular, \(\nabla\widetilde K_h\in L^\infty(\mathbb T^d)\), and
    \(\nabla\widetilde K_h\) is globally Lipschitz.

    \item For every \(v\in H^1(\mathbb T^d)\),
    \[
    \|\nabla v-\nabla(\widetilde K_h*_{\mathbb T}v)\|_{L^2(\mathbb T^d)}
    \to 0
    \qquad\text{as } h\downarrow0.
    \]
    Moreover, if \(v\in H^2(\mathbb T^d)\), then there exists a constant
    \(C>0\), independent of \(h\) and \(v\), such that for \(0<h\le1\),
    \[
    \|\nabla v-\nabla(\widetilde K_h*_{\mathbb T}v)\|_{L^2(\mathbb T^d)}
    \le
    Ch\|D^2v\|_{L^2(\mathbb T^d)}.
    \]
\end{enumerate}
\end{lemma}
The proof is standard and is given in Appendix~\ref{app:proof_kernel} for completeness. 
Although Lemma~\ref{lemma:kernel} is stated on the periodic domain
\(\mathbb T^d\), the periodicity assumption is not essential for these
kernel estimates. In the non-periodic whole-space setting, the same
conclusions hold on \(\mathbb R^d\) for any even, normalized
\(C_c^\infty(\mathbb R^d)\) mollifier \(K\), with the usual scaling
\(K_h(\bx)=h^{-d}K(\bx/h)\) and \(\widetilde K_h=K_h*K_h\). 
The proof follows from the standard approximate-identity and translation
estimate arguments, and we omit the routine variant for brevity.
Based on the properties of the kernel, we have the following well-posedness result.
\begin{assumption}\label{ass:regularity}
Assume that the $\bb_{\mathrm{model}}$ and $\bb_{\mathrm{true}}$ satisfies:
\medskip
\noindent{(A1)}
There exists a constant $B_{\mathrm{model}}>0$ such that for all $\rho\in\mathcal P_2(\Omega)$,
\[
\|\bbm(\cdot,\rho)\|_{L^\infty(\Omega)} \le B_{\mathrm{model}},
\qquad
\|\nabla\cdot \bbm(\cdot,\rho)\|_{L^\infty(\Omega)} \le B_{\mathrm{model}} .
\]

\medskip
\noindent{(A2)}
Assume that $\bb_{\true}$ and $\bb_{\model}$ is globally Lipschitz, i.e. there exists a constant $L_\true,L_\model>0$ such that for all $x,y\in\Omega$ and all $\rho,\rho'\in\mathcal P_2(\Omega)$,
\[
\|\bbt(\bx_1,\rho_1)-\bbt(\bx_2,\rho_2)\| \le L_{\true}(\|\bx_1-\bx_2\|+W_2(\rho_1,\rho_2)),
\]
and
\[
\|\bbm(\bx_1,\rho_1)-\bbm(\bx_2,\rho_2)\| \le L_{\model}(\|\bx_1-\bx_2\|+W_2(\rho_1,\rho_2)) .
\]

\end{assumption}
We now turn to the well-posedness of the regularized assimilated dynamics.
Although the nudging term depends nonlocally on the law $\nu_t$, convolution
against $\tilde K_h$ renders it Lipschitz in both the state and the measure, so
the standard McKean-Vlasov well-posedness theory applies.
\begin{proposition}[Well-Posedness of Equation \eqref{equ:assimilated_regularized}]\label{prop:well_possedness}
Under (A2), and also assume that $\nabla \tilde{K}_h$ is globally Lipschitz, i.e. there exists $L>0$ such that for any $\bx_1,\bx_2\in\mathbb{R}^d$, 
\begin{equation*}
\|\nabla\tilde{K}_h(\bx_1)-\nabla\tilde{K}_h(\bx_2)\|\leq L\|\bx_1-\bx_2\|.
\end{equation*}
For any $T>0$ and $ \Law(\bZ_0) \in \mathcal P_2(\mathbb R^d)$, the SDE \eqref{equ:assimilated_regularized} has a unique strong solution on $[0,T]$ and consequently, its law is the unique solution to the Fokker-Planck equation \eqref{equ:fpk_regularized}.
\end{proposition}
The proof is given in Appendix \ref{app:proof_well_possedness}. Next, we approximate the assimilated mean-field process by finite particles as
$$
\begin{aligned}
\intd \bZ^{i,N}_t = &\bb_{\mathrm{model}}(\bZ^{i,N}_t, \nu^N_t) \intd t \\
&-  \lambda \nabla \left[ \left(\tilde  K_h * \nu^N_t\right)(\bZ^{i,N}_t) -\left( K_h * \mu^{\mathrm{obs}}_t\right)(\bZ^{i,N}_t) \right] \intd t
+ \Sigma \intd \bW^i_t, 
\end{aligned}
$$
where $\nu^N_t = \frac{1}{N}\sum_{j=1}^N \delta_{\bZ^{j,N}_t}$ is the empirical measure of the system. Expanding the convolution terms, the explicit interaction dynamics are given by:
\begin{equation}\label{equ:density_obs}
\begin{aligned}
\intd \bZ^{i,N}_t
=& \bb_{\mathrm{model}}\left(\bZ^{i,N}_t, \nu^N_t \right)\,\intd t  \\ 
& - \lambda\Bigg(\frac{1}{N}\sum_{j=1}^N \nabla \tilde  K_h(\bZ^{i,N}_t - \bZ^{j,N}_t)  -  \nabla   K_h *\mu^{\mathrm{obs}} (\bZ^{i,N}_t)\Bigg)\,\intd t+ \Sigma \,\intd \bW^i_t.
\end{aligned}
\end{equation}
If we only observe the unlabeled  locations $\{\bX_t^{\mathrm{obs},k,M}\}_{k=1}^M$, then
\begin{equation}\label{equ:location_obs}
\begin{aligned}
\intd \bZ^{i,N}_t
=& \bb_{\mathrm{model}}\left(\bZ^{i,N}_t, \nu^N_t \right)\,\intd t\\
\quad\,
&  - \lambda \Bigg(\frac{1}{N}\sum_{j=1}^N \nabla \tilde  K_h(\bZ^{i,N}_t - \bZ^{j,N}_t) 
 - \frac{1}{M}\sum_{k=1}^M \nabla   \tilde K_h(\bZ^{i,N}_t - \bX_t^{\mathrm{obs},k,M})\Bigg)\,\intd t \\
 &+ \Sigma \,\intd \bW^i_t.
\end{aligned}
\end{equation}
\begin{proposition}[Mean-Field Convergence and Propagation of Chaos for the Nudging Particle System]\label{equ:propagation_of_chaos}
    Let the assumptions of Proposition \ref{prop:well_possedness} hold. Also, assume that $\nabla \tilde{K}_h$ is bounded, i.e. $\|\nabla \tilde{K}_h\|_\infty < \infty$. Let $(\bZ_t^{i,N})_{i=1}^N$ be the $N$-particle system solving \eqref{equ:density_obs} with $f_0$-chaotic initial data $\bZ_0^{i,N}\sim f_0$. Let $f_t$ be the unique solution to the mean field Fokker-Planck equation \eqref{equ:fpk_regularized} with initial condition $f_0$. Then, the $N$-particle system \eqref{equ:density_obs} converges to mean field model \eqref{equ:assimilated_regularized} as $N\to \infty$. That is, for any $T>0$, the $N$-particle distribution $f_t^N= \mathrm{Law}(\bZ_t^{1,N},\cdots,\bZ_t^{N,N})$ is $f_t$-chaotic, satisfying:
    \[
    \lim_{N\to\infty} W_2\left(f_{[0,T]}^{1,N},f_{[0,T]}\right) = 0,
    \]
    where $f_t^{1,N}$ is the first marginal of $f_t^N$.
\end{proposition}
The proof appears in Appendix \ref{app:propagation_of_chaos}. 

We next show that the solution of equation~\eqref{equ:assimilated_regularized}
remains strictly positive on $[0,T]$ for any fixed final time $T>0$,
a property that will be needed in the proof of convergence.

\begin{proposition}[Positivity on finite time intervals]
\label{prop:positivity}
Let $T>0$ be fixed and let $\Omega=\mathbb T^d$. Assume that
Assumption~\ref{ass:regularity} holds, and let $K_h$ be the Gaussian kernel
introduced in Lemma~\ref{lemma:kernel}. Let $\nu$ be a classical
solution on $[0,T]$ of
\[
\partial_t \nu
= \nabla\cdot(A\nabla \nu)-\nabla\cdot\bigl(\nu\,c[\nu]\bigr),
\]
where
\[
c[\nu](x,t)
:= \bb_{\mathrm{model}}(x,\nu_t)
-\lambda \nabla\Bigl(K_h*(K_h*\nu_t-\mu_t^{\mathrm{obs}})\Bigr)(x).
\]
Assume moreover that
\[
\sup_{t\in[0,T]}
\Bigl(
\|K_h*\nu_t\|_{L^\infty(\Omega)}
+\|\mu_t^{\mathrm{obs}}\|_{L^\infty(\Omega)}
\Bigr)
<\infty,
\]
and that the initial density satisfies
\[
\nu_0(x)\ge \underline\nu_0>0,
\qquad \forall x\in\Omega.
\]
Then there exists a constant $\underline\nu_T>0$ such that
\[
\nu_t(x)\ge \underline\nu_T,
\qquad \forall (x,t)\in\Omega\times[0,T].
\]
More precisely,
\[
\underline\nu_T=\underline\nu_0 e^{-CT},
\]
where
\[
C:=B_{\mathrm{model}}
+\lambda \|\Delta K_h\|_{L^1(\Omega)}
\sup_{t\in[0,T]}
\|K_h*\nu_t-\mu_t^{\mathrm{obs}}\|_{L^\infty(\Omega)}.
\]
\end{proposition}
The proof is provided in Appendix \ref{app:proof_positivity}.

The following result should be interpreted as an idealized contraction estimate under exact smoothed observations. Finite-sample and noisy observations add additional residual terms to the bias floor.
\begin{theorem}[$L^2$-error decay under kernel-regularized nudging]
\label{thm:main_kdenudge}
Let $\Omega=\mathbb T^d$ and let $A=\frac12\Sigma\Sigma^\top$ be a constant, symmetric positive-definite matrix with smallest eigenvalue $\kappa>0$.
Let $(\mu_t)_{t\ge0}$ and $(\nu_t)_{t\ge0}$ be classical solutions of the true Fokker--Planck equation~\eqref{equ:true_density} and the assimilated equation~\eqref{equ:fpk_regularized}, respectively, with initial conditions satisfying $\displaystyle\int_\Omega (\nu_0-\mu_0)\,\intd x = 0$. Define the error $e_t:=\nu_t-\mu_t$ and the $L^2$-error energy
\[
V(t) := \tfrac{1}{2}\|e_t\|_{L^2(\Omega)}^2.
\]
Suppose the following conditions hold.
\begin{enumerate}[label=\textup{(C\arabic*)},leftmargin=*]
    \item \label{cond:lip_L2}
    The model drift satisfies Assumption~\ref{ass:regularity}, and, in addition,
    \[
    \|\bbm(\cdot,\mu)-\bbm(\cdot,\nu)\|_{L^2(\Omega)}
    \le L_\mu\,\|\mu-\nu\|_{L^2(\Omega)}
    \]
    for all $\mu,\nu\in\mathcal P_2(\Omega)$.
 
    \item \label{cond:density_bounds}
    The true density and the assimilated density are uniformly bounded above, 
    \[
    \|\mu_t\|_{L^\infty},\|\nu_t\|_{L^\infty}\le\bar\rho\], and the assimilated density is uniformly bounded below, $\nu_t(x)\ge\underline\nu>0$, for all $(x,t)\in\Omega\times[0,T]$. (The lower bound is guaranteed by Proposition~\ref{prop:positivity} in some special cases.)
 
    \item \label{cond:kernel_approx}
    There exists a function $\delta(h)\to 0$ as $h\to 0$ such that
    \[
    \|\nabla e_t - \nabla(\tilde K_h*e_t)\|_{L^2(\Omega)}
    \le \delta(h)\,\|\nabla e_t\|_{L^2(\Omega)}.
    \]
\end{enumerate}
 
\medskip\noindent
Define the advection instability coefficient
\[
C_{\mathrm{adv}} := B_{\mathrm{model}} + \bar\rho\,L_\mu,
\]
the model-error bound
\[
\Delta := \sup_{t\in[0,T]}\,
\|\bbm(\cdot,\mu_t)-\bbt(\cdot,\mu_t)\|_{L^\infty(\Omega)},
\]
and the Poincar\'e constant $C_P$ of $\Omega$ for zero-mean functions.
If the observation resolution $h$ and the nudging intensity $\lambda$ satisfy
\begin{equation}\label{eq:h_condition}
\delta(h) \le \frac{\underline\nu}{2\bar\rho},
\end{equation}
\begin{equation}\label{eq:lambda_condition}
\lambda > \frac{1}{\underline\nu}
\left( \frac{2C_{\mathrm{adv}}^2\,C_P}{\kappa} - \kappa \right),
\end{equation}
then the $L^2$ error decays exponentially to a model-error floor. Explicitly,
with the rate
\[
\alpha
:=
\frac{\kappa+\lambda\underline\nu}{C_P}
-\frac{2C_{\mathrm{adv}}^2}{\kappa}
\;>\;0,
\]
it holds for all $t\ge 0$ that
\begin{equation}\label{eq:L2_bound}
V(t)\le e^{-\alpha t}\,V(0)+\frac{\bar\rho^2}{\alpha\kappa}\,\Delta^2 .
\end{equation}
In particular,
\[
\limsup_{t\to\infty}\|\nu_t-\mu_t\|_{L^2(\Omega)}^2
\le \frac{2\bar\rho^2}{\alpha\kappa}\,\Delta^2,
\]
so the long-time error is controlled by a bias floor proportional to $\Delta^2$
that vanishes when the model is exact ($\Delta=0$).

\end{theorem}
 
\begin{proof}
Subtracting the true equation~\eqref{equ:true_density} from the assimilated equation~\eqref{equ:fpk_regularized}, the error $e_t=\nu_t-\mu_t$ satisfies
\begin{equation}\label{equ:error_evolution}
\partial_t e_t
= \nabla\cdot(A\nabla e_t)
- \nabla\cdot\bigl(\nu_t\bbm(\cdot,\nu_t)\bigr)
+ \nabla\cdot\bigl(\mu_t\bbt(\cdot,\mu_t)\bigr)
+ \lambda\,\nabla\cdot\bigl(\nu_t\,\nabla(K_h*r_t)\bigr),
\end{equation}
where $r_t = K_h*\nu_t - \mu_t^{\mathrm{obs}} = K_h*e_t$.
Since~\eqref{equ:error_evolution} is in divergence form, the zero-mean condition is preserved:
$\int_\Omega e_t\,\intd x = 0$ for all $t\ge 0$.
Hence the Poincar\'e inequality applies:
\begin{equation}\label{eq:poincare_e}
\|e_t\|_{L^2(\Omega)}^2 \le C_P\,\|\nabla e_t\|_{L^2(\Omega)}^2.
\end{equation}
Testing~\eqref{equ:error_evolution} against $e_t$ and integrating by parts yields
\begin{equation}\label{eq:energy_identity}
\begin{aligned}
  V'(t)
=& \underbrace{-\langle \nabla e_t,\, A\nabla e_t\rangle}_{I_{\mathrm{diff}}}
\;+\; \underbrace{\langle \nabla e_t,\, \nu_t\bbm(\cdot,\nu_t) - \mu_t\bbt(\cdot,\mu_t)\rangle}_{I_{\mathrm{adv}}}\\
&
- \underbrace{\lambda\langle \nabla e_t,\, \nu_t\,\nabla(K_h*r_t)\rangle}_{I_{\mathrm{nud}}}.  
\end{aligned}
\end{equation}
For the diffusion term, by the coercivity of $A$,
\[
I_{\mathrm{diff}} = -\langle\nabla e_t,\,A\nabla e_t\rangle
\le -\kappa\,\|\nabla e_t\|^2.
\]
Next, for the advection term, we decompose the flux difference as
\[
\begin{aligned}
\nu_t\bbm(\cdot,\nu_t) - \mu_t\bbt(\cdot,\mu_t)
= & e_t\,\bbm(\cdot,\nu_t)
+ \mu_t\bigl(\bbm(\cdot,\nu_t)-\bbm(\cdot,\mu_t)\bigr)\\
&+ \mu_t\bigl(\bbm(\cdot,\mu_t)-\bbt(\cdot,\mu_t)\bigr).  
\end{aligned}
\]
Using (C1), (C2), and Cauchy--Schwarz, each inner product with $\nabla e_t$ is bounded by
\[
I_{\mathrm{adv}}
\le B_{\mathrm{model}}\|\nabla e_t\|\,\|e_t\|
+ \bar\rho L_\mu\|\nabla e_t\|\,\|e_t\|
+ \bar\rho\,\|\nabla e_t\|\,\Delta
= C_{\mathrm{adv}}\|\nabla e_t\|\,\|e_t\|
+ \bar\rho\,\|\nabla e_t\|\,\Delta.
\]
Applying Young's inequality $ab\le\frac{\epsilon}{2}a^2+\frac{1}{2\epsilon}b^2$ to each product with a parameter $\epsilon>0$ gives
\begin{equation}\label{eq:adv_young}
I_{\mathrm{adv}}
\le \epsilon\,\|\nabla e_t\|^2
+ \frac{C_{\mathrm{adv}}^2}{2\epsilon}\,\|e_t\|^2
+ \frac{\bar\rho^2}{2\epsilon}\,\Delta^2.
\end{equation}
 
Lastly, we bound the nudging term. Since $K_h*r_t = \tilde K_h*e_t$, we add and subtract $\nabla e_t$ inside the inner product:
\[
\begin{aligned}
I_{\mathrm{nud}}
&= -\lambda\langle\nu_t\nabla e_t,\,\nabla(\tilde K_h*e_t)\rangle \\
&= -\lambda\langle\nu_t\nabla e_t,\,\nabla e_t\rangle
   -\lambda\langle\nu_t\nabla e_t,\,\nabla(\tilde K_h*e_t)-\nabla e_t\rangle.
\end{aligned}
\]
The first term is bounded using (C2): $-\lambda\langle\nu_t\nabla e_t,\nabla e_t\rangle \le -\lambda\underline\nu\,\|\nabla e_t\|^2$.
For the second, Cauchy--Schwarz and the kernel-approximation condition~(C3) yield
\[
\bigl|\lambda\langle\nu_t\nabla e_t,\,\nabla(\tilde K_h*e_t)-\nabla e_t\rangle\bigr|
\le \lambda\bar\rho\,\|\nabla e_t\|\,\delta(h)\|\nabla e_t\|
= \lambda\bar\rho\,\delta(h)\,\|\nabla e_t\|^2.
\]
Combining,
\begin{equation}\label{eq:nud_est}
I_{\mathrm{nud}}
\le -\lambda\bigl(\underline\nu - \bar\rho\,\delta(h)\bigr)\|\nabla e_t\|^2.
\end{equation}
Substituting~\eqref{eq:adv_young} and~\eqref{eq:nud_est} into~\eqref{eq:energy_identity} gives
\[
V'(t)
\le \bigl(-\kappa - \lambda(\underline\nu-\bar\rho\,\delta(h)) + \epsilon\bigr)\|\nabla e_t\|^2
+ \frac{C_{\mathrm{adv}}^2}{2\epsilon}\|e_t\|^2
+ \frac{\bar\rho^2}{2\epsilon}\Delta^2.
\]
We now make two parameter choices.
First, condition~\eqref{eq:h_condition} gives $\underline\nu - \bar\rho\,\delta(h) \ge \underline\nu/2$, so
\[
-\kappa - \lambda(\underline\nu-\bar\rho\,\delta(h)) + \epsilon
\le -\kappa - \tfrac{\lambda\underline\nu}{2} + \epsilon.
\]
Second, choosing $\epsilon = \kappa/2$ to retain half of the diffusive
dissipation yields
\[
V'(t)
\le -\bigl(\tfrac{\kappa}{2}+\tfrac{\lambda\underline\nu}{2}\bigr)\|\nabla e_t\|^2
+ \frac{C_{\mathrm{adv}}^2}{\kappa}\|e_t\|^2
+ \frac{\bar\rho^2}{\kappa}\Delta^2.
\]
Applying the Poincar\'e inequality~\eqref{eq:poincare_e} to the gradient term,
\[
V'(t)
\le -\underbrace{\left[\frac{1}{C_P}\Bigl(\frac{\kappa}{2}+\frac{\lambda\underline\nu}{2}\Bigr)
- \frac{C_{\mathrm{adv}}^2}{\kappa}\right]}_{=:\,\frac12\alpha}\|e_t\|^2
+ \frac{\bar\rho^2}{\kappa}\Delta^2.
\]
Condition~\eqref{eq:lambda_condition} ensures $\alpha>0$. Since $\|e_t\|^2 = 2V(t)$, we obtain
\[
V'(t) \le -\alpha\,V(t) + \frac{\bar\rho^2}{\kappa}\Delta^2.
\]
The bound~\eqref{eq:L2_bound} follows by Gronwall's inequality.
\end{proof}

Condition~\eqref{eq:h_condition} requires the observation bandwidth $h$ to be small enough so that the smoothing error does not overwhelm the density lower bound. Condition~\eqref{eq:lambda_condition} requires the nudging intensity $\lambda$ to be large enough to overcome the advection instability caused by model error and nonlinear transport. Together, they guarantee exponential convergence of the full $L^2$-error $\|\nu_t-\mu_t\|_{L^2}$ to a neighborhood of zero whose radius is controlled by the model bias~$\Delta$.

\section{Numerical experiments}

We evaluate the proposed multiscale nudging method on five examples of increasing complexity: a linear Gaussian benchmark, a bistable bimodal system, a mean-field Lorenz model, Vlasov–Poisson, and a real collective-motion dataset. In each case, we compare three evolutions: the reference dynamics, a biased forecast model without assimilation, and the assimilated forecast obtained by adding the nudging term. 

\subsection{Simple linear system}

As a first test, we consider the one-dimensional linear interacting particle system
\begin{equation}
    \intd X_t^i = -a\bigl(X_t^i - m_t\bigr)\,\intd t + \intd W_t^i,
\end{equation}
where
\[
m_t = \frac{1}{N}\sum_{i=1}^N X_t^i
\]
is the empirical mean. As \(N \to \infty\), the associated mean-field limit is given by
\[
\intd \bar X_t = -a\bigl(\bar X_t - \bar m_t\bigr)\,\intd t + \intd W_t,
\qquad
\bar m_t = \int x\,\mu_t(\intd x),
\]
where \(\mu_t = \mathrm{Law}(\bar X_t)\). The corresponding Fokker--Planck equation reads
\[
\partial_t \mu_t
= \partial_x\left(a\left(x-\int x\,\mu_t(\intd x)\right)\mu_t\right)
+ \frac{1}{2}\partial_{xx}\mu_t.
\]

This example is useful as a first test since the true law remains close to Gaussian and its variance provides an example of the distributional error. We generate the reference data using the true interaction coefficient \(a=1\) with initial condition \(X_0^i \sim \mathcal{N}(0,0.5)\). The particle system is integrated by the Euler--Maruyama scheme with time step \(0.01\) up to \(T=5\). Observations are produced from the smoothed density \(K_h * \mu_t\) with bandwidth \(h=0.5\). To test robustness with respect to model bias, the forecast model uses the same dynamics but with \(a \in \{0.5,2,5\}\).

Figure~\ref{fig:simple_linear_case_a_0.5} reports the variance of the trajectories for the under-interacting case \(a=0.5\). The four panels correspond to different numbers of nudging substeps. Without assimilation, the forecast systematically underestimates the growth of the variance and drifts away from the reference solution. When \(\lambda=1\), the feedback is too weak to noticeably change the forecast. Increasing the nudging strength to \(10\) and \(100\)  reduces the variance gap. The strongest correction, \(\lambda=1000\), gives the closest match once enough nudging substeps are applied, but with only one nudging update, it produces a sharp overshoot, illustrating the expected stability-accuracy trade-off. The same qualitative behavior persists for the over-interacting cases \(a=2\) and \(a=5\); see Appendix Figures~\ref{fig:simple_linear_case_a_2} and~\ref{fig:simple_linear_case_a_5}.

\begin{figure}[htbp]
    \centering
    \includegraphics[width=0.65\textwidth]{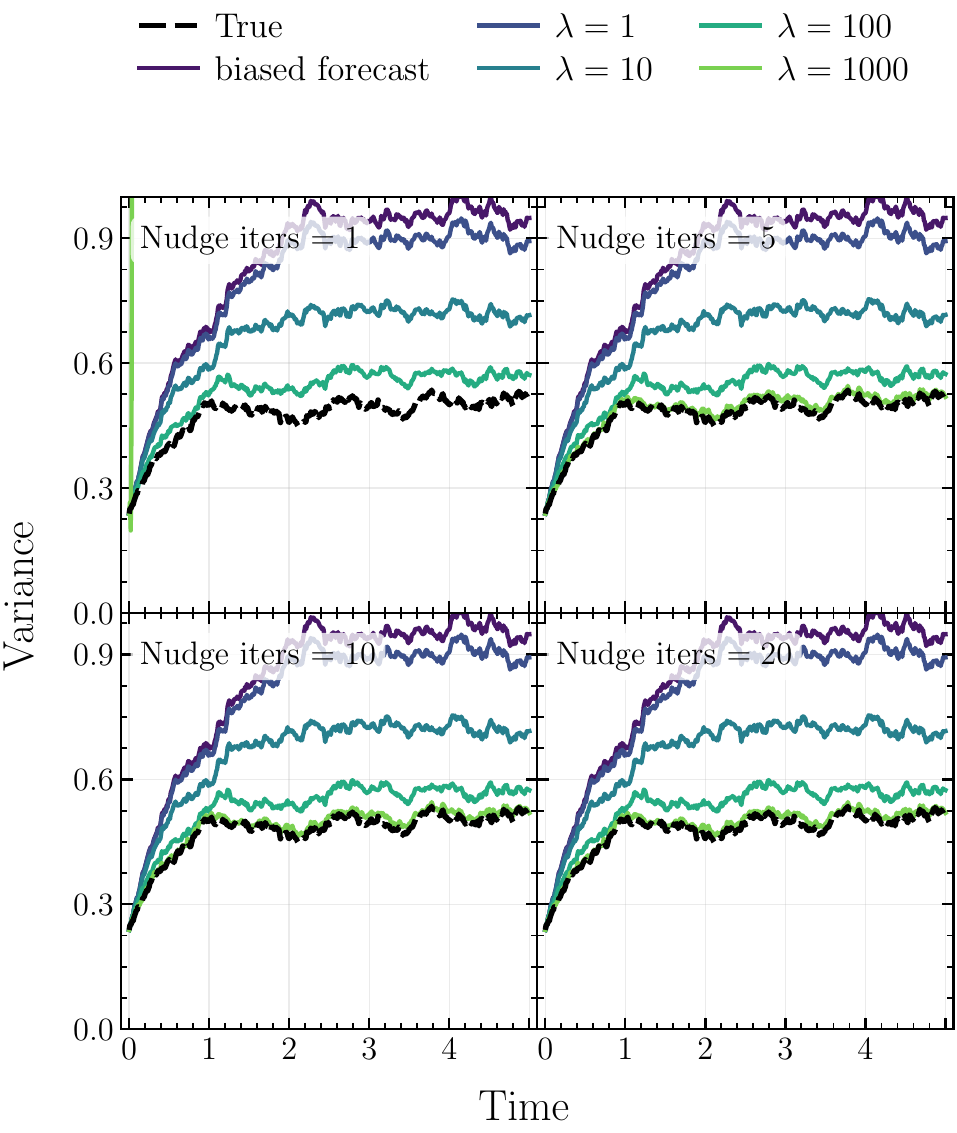}
    \caption{\textbf{Variance dynamics in the one-dimensional linear benchmark (\(a=0.5\)).} We compare the reference system, the biased forecast model, and assimilated (nudged) trajectories with \(\lambda\in\{1,10,100,1000\}\). The four panels correspond to different numbers of nudging substeps. Increasing \(\lambda\) improves tracking accuracy, while excessively large nudging can introduce temporary numerical instability when the correction is applied too aggressively.}
    \label{fig:simple_linear_case_a_0.5}
\end{figure}

To move beyond a single moment diagnostic, Figure~\ref{fig:simple_linear_case_density_a_0.5} compares the full space-time density for the true dynamics, the biased forecast, and two assimilated solutions. The biased forecast is overly diffuse and fails to reproduce the concentration near \(x=0\). A moderate correction (\(\lambda=10\)) partially restores the correct density profile, while the stronger correction (\(\lambda=1000\)) yields a space-time distribution much closer to the reference one. This confirms that the nudging term corrects the full law.

\begin{figure}[h]
    \centering
    \includegraphics[width=0.65\textwidth]{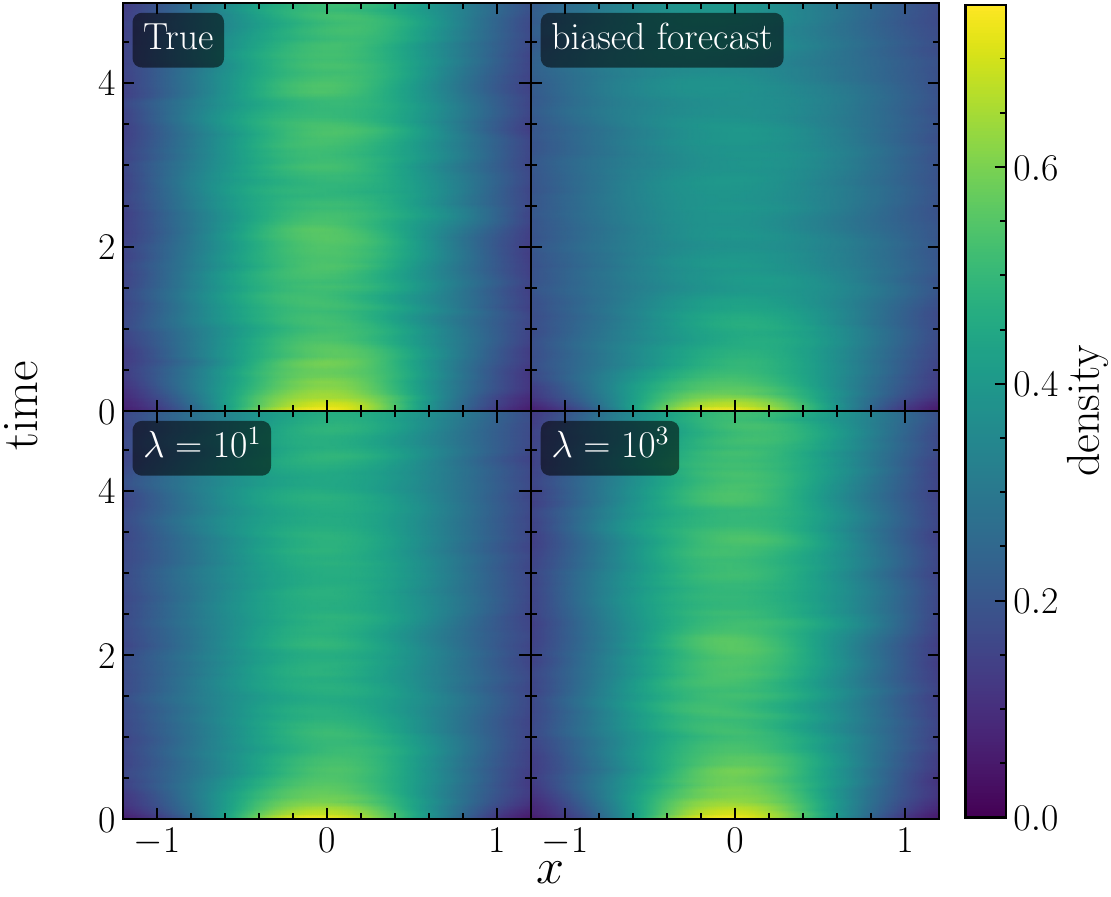}
    \caption{\textbf{Space--time density evolution in the linear benchmark (\(a=0.5\)).} Top-left: reference density. Top-right: biased forecast. Bottom-left: assimilated density with \(\lambda=10\). Bottom-right: assimilated density with \(\lambda=1000\). Larger nudging strength restores both the location and the spread of the true law.}
    \label{fig:simple_linear_case_density_a_0.5}
\end{figure}

Figure~\ref{fig:simple_linear_case_w2} summarizes the final-time and time-averaged \(W_2\) errors for all three biased models as functions of \(\lambda\). Both error metrics decrease nearly monotonically as the nudging strength increases, with the most visible gains occurring between \(10\) and \(100\). The reduction is consistent across both under-interacting and over-interacting forecast models, which supports the theoretical conclusion of Theorem~\ref{thm:main_kdenudge}: sufficiently strong observation feedback suppresses the error induced by model misspecification.

\begin{figure}
    \centering
    \includegraphics[width=0.85\linewidth]{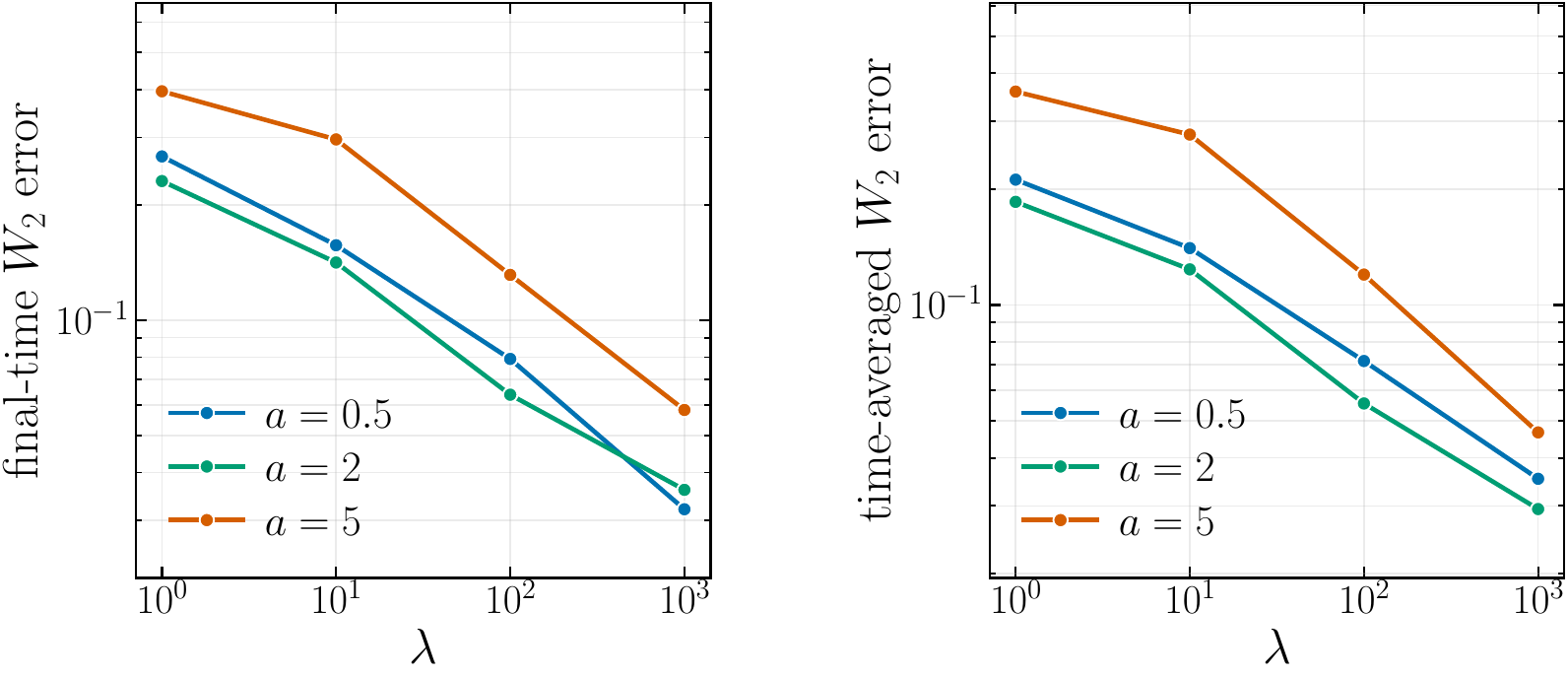}
    \caption{\textbf{Distributional error versus nudging strength in the linear benchmark.} Left: final-time \(W_2\) error. Right: time-averaged \(W_2\) error. In all three biased models \(a\in\{0.5,2,5\}\), stronger nudging produces a smaller distributional discrepancy with the reference dynamics.}
    \label{fig:simple_linear_case_w2}
\end{figure}

\subsection{Multimodal distribution}

We next consider a one-dimensional interacting particle system with a bistable confining potential,
\begin{equation}
    \intd X_t^i = -\nabla U(X_t^i)\,\intd t -a\bigl(X_t^i - m_t\bigr)\,\intd t + \intd W_t^i,
\end{equation}
where
\[
U(x) = \frac{x^4}{4} -\frac{x^2}{2}, \]
and where the empirical mean is
\[
m_t = \frac{1}{N}\sum_{i=1}^N X_t^i.
\]
As \(N \to \infty\), the associated mean-field limit is given by
\[
\intd \bar X_t = - (\bar{X}_t^3 - \bar{X}_t )\,\intd t -a\bigl(\bar X_t - \bar m_t\bigr)\,\intd t + \intd W_t,
\qquad
\bar m_t = \int x\,\mu_t(\intd x),
\]
where \(\mu_t = \mathrm{Law}(\bar X_t)\). The corresponding Fokker--Planck equation reads
\[
\partial_t \mu_t
= \partial_x\left( (x^3-x )\mu_t\right)
+ \partial_x\left( a\left(x-\int x\,\mu_t(\intd x)\right)\mu_t\right)
+ \frac{1}{2}\partial_{xx}\mu_t.
\]

This benchmark is more challenging than the linear Gaussian case since the double-well potential generates a  bimodal law with metastable transitions between the two wells. In this regime, matching only the mean or the variance is insufficient; two distributions can have similar low-order moments while placing mass in the wrong well. We therefore use this example to test whether the kernel-based nudging term can transfer probability mass across the barrier at \(x=0\), recover the correct modal locations, and preserve the relative weights of the two modes from smoothed macroscopic observations. In the experiments, we generate the reference data with the true interaction coefficient $a=0.25$ and initial position $X^i_0 \sim \mathcal{N}(0,0.5)$. We integrate the particle system using Euler-Maruyama with time step $\Delta t=0.01$ up to final time $T=5$. As before, observations are generated from the smoothed density $K_h*\mu_t$ with $h=0.5$, and the forecast model is misspecified by taking the interaction coefficient \(a\in\{0.1,0.5,1.5\}\). The quality of assimilation is evaluated through density snapshots together with the corresponding Wasserstein error. This example highlights that the proposed correction mechanism is not restricted to unimodal or near-Gaussian laws.

\begin{figure}[htbp]
    \centering
    \includegraphics[width=0.65\textwidth]{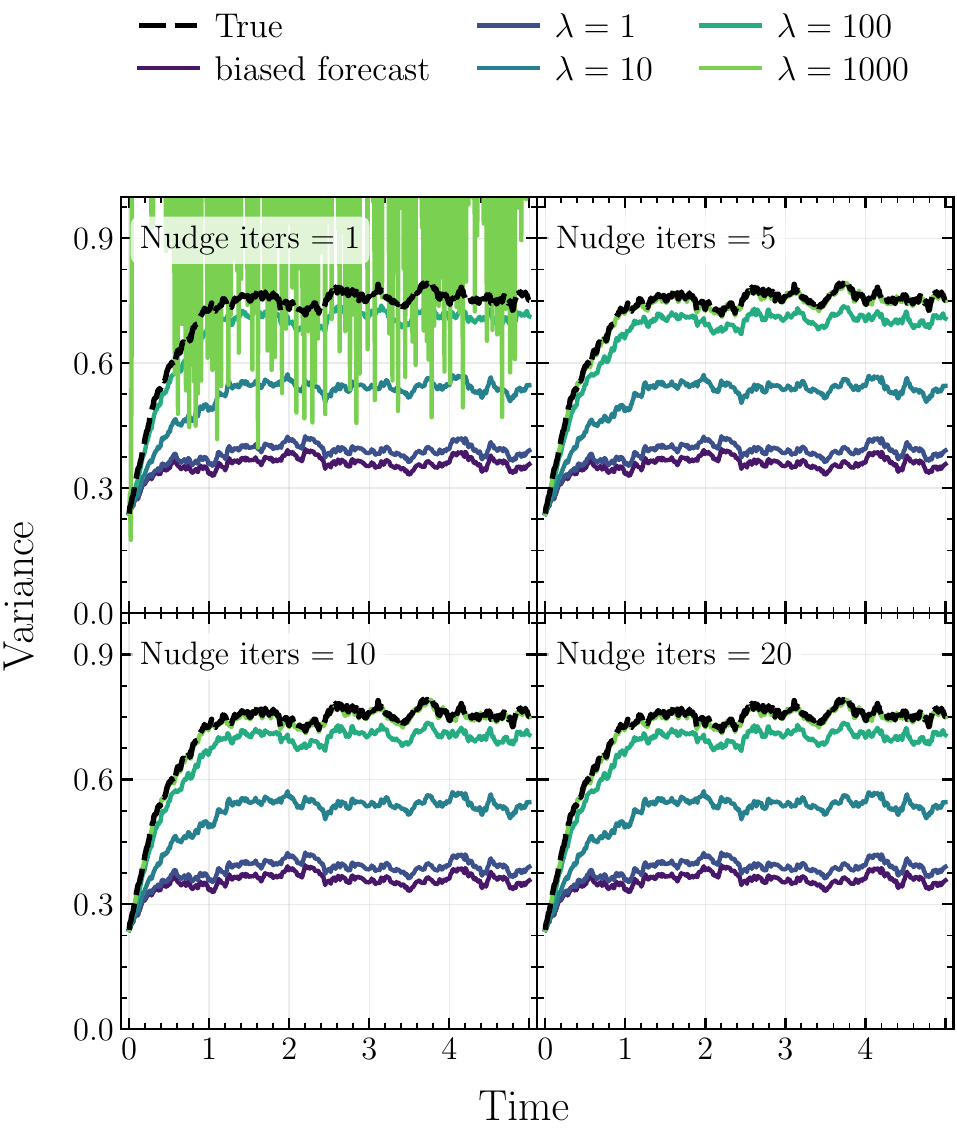}
    \caption{\textbf{Variance dynamics in the multimodal benchmark (\(a=1.5\)).} We compare the reference system, the biased forecast model, and assimilated (nudged) trajectories with \(\lambda\in\{1,10,100,1000\}\). The four panels correspond to different numbers of nudging substeps. The same stability-accuracy trade-off observed in Figure  ~\ref{fig:simple_linear_case_a_0.5} persists in this regime.}
    \label{fig:case2_true_0.25_a_1.5}
\end{figure}

Figure \ref{fig:case2_true_0.25_a_1.5} shows the variance dynamics for $a=1.5$ for four different nudging substeps. The biased forecast model performs worse than all assimilated forecast models. After enough nudging substeps, the assimilated model with nudging intensity $\lambda = 10^3$ produces the closest approximation to the reference solution, but with just one nudging iteration, the model overestimates the variance and generates large oscillations, as expected due to the trade-off between accuracy and stability.

\begin{figure}[h]
    \centering
    \includegraphics[width=0.65\textwidth]{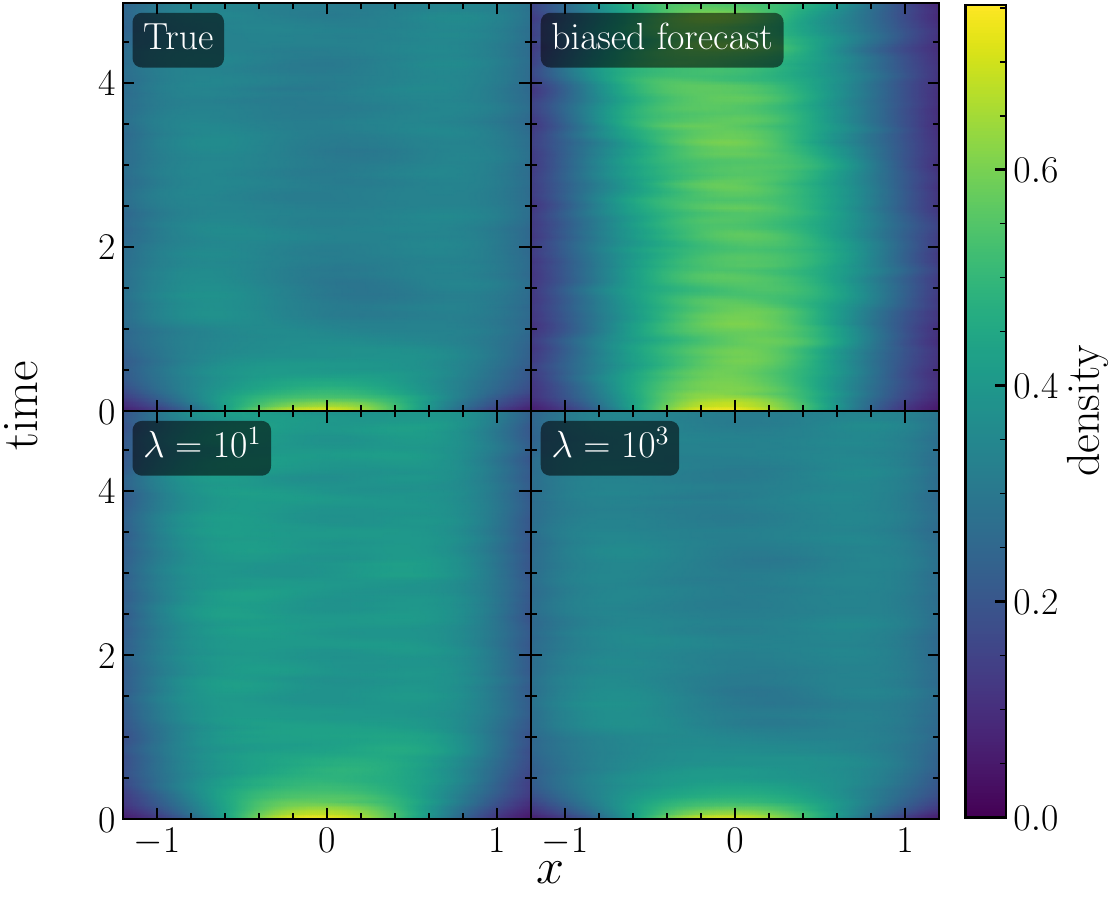}
    \caption{\textbf{Space--time density evolution in the multimodal benchmark (\(a=1.5\)).} Top-left: reference density. Top-right: biased forecast. Bottom-left: assimilated density with \(\lambda=10\). Bottom-right: assimilated density with \(\lambda=1000\). Larger nudging strength restores both the location and the spread of the true law.}
    \label{fig:case2_density_true_0.25_a_1.5}
\end{figure}

In Figure \ref{fig:case2_density_true_0.25_a_1.5}, the space-time densities for the true reference solution, biased forecast, and two assimilated forecasts with $a=1.5$ are shown. The prediction from the biased model is overly concentrated around $x=0$ and does not capture the density profile of the true solution, which has two high probability regions. The assimilated model with $\lambda = 10$ moderately diffuses the concentration near $x=0$ and partially resolves the two modes, whereas the model with $\lambda = 10^3$ achieves the closest match to the reference density and clearly shows the two stable regions of high probability.

\begin{figure}
    \centering
\includegraphics[width=0.85\textwidth]{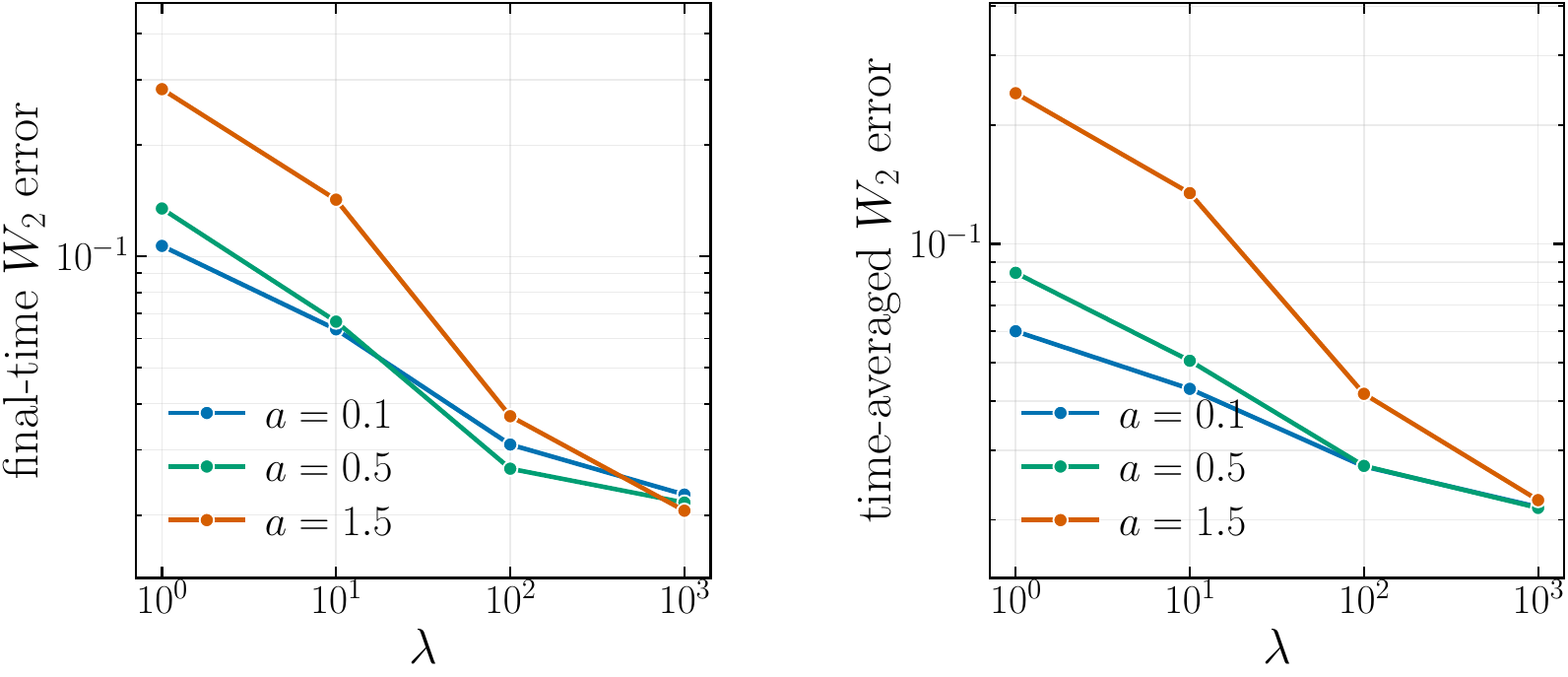}
    \caption{\textbf{Distributional error versus nudging strength in the multimodal benchmark.} Left: final-time \(W_2\) error. Right: time-averaged \(W_2\) error. In all three biased models \(a\in\{0.1,0.5,1.5\}\), stronger nudging produces a smaller distributional discrepancy with the reference dynamics.}
    \label{fig:case2_true_0.25_w2}
\end{figure}

Figure \ref{fig:case2_true_0.25_w2} displays the final-time and time-averaged $W_2$ errors for the forecast models with $a\in \{0.1,0.5,1.5\}$ as functions of the nudging intensity $\lambda$. As $\lambda$ increases, both errors decrease monotonically across all three biased models, which is consistent with the conclusion of Theorem~\ref{thm:main_kdenudge}. 
Additional experiments with $a=0.1$ and $a=0.5$ show similar qualitative behavior; see Appendix Figures \ref{fig:case2_true_0.25_a_0.1}, \ref{fig:case2_true_0.25_a_0.5}, \ref{fig:case2_density_true_0.25_a_0.1}, and \ref{fig:case2_density_true_0.25_a_0.5}.

\subsection{Lorenz dynamics}

We further consider a mean-field Lorenz system
\[
\begin{aligned}
          \intd X^i_t =& \sigma_L (m_{t,y} - X^i_t)\,\intd t + \sigma \,\intd W_t^{x,1},\\
         \intd Y^i_t =& [m_{t,x} (\rho - m_{t,z}) - Y^i_t]\,\intd t + \sigma \,\intd W_t^{y,1},\\
         \intd Z^i_t =& [m_{t,x} m_{t,y} - \beta Z^i_t]\,\intd t + \sigma \,\intd W_t^{z,1},   
\end{aligned}
\]
where \(m_{t,x}\), \(m_{t,y}\), and \(m_{t,z}\) denote the empirical means of the three coordinates. As a biased forecast model, we use \(N\) individual Lorenz systems without the mean-field coupling,
\[
\begin{aligned}
          \intd X^i_t =& \sigma_L (Y^i_t - X^i_t)\,\intd t + \sigma \,\intd W_t^{x,1},\\
         \intd Y^i_t =& [X^i_t  (\rho - Z^i_t) - Y^i_t]\,\intd t + \sigma \,\intd W_t^{y,1},\\
         \intd Z^i_t =& [X^i_t  Y^i_t - \beta Z^i_t]\,\intd t + \sigma \,\intd W_t^{z,1}.   
\end{aligned}
\]
We use the classical Lorenz parameters
\[
\sigma_L=10,\qquad \rho=28,\qquad \beta=8/3,
\]
with noise amplitude \(\sigma_{\rm noise}=1\), time step \(\Delta t=0.01\), particle number \(N=1000\), and kernel bandwidth \(h=0.5\).
This example probes a strongly nonlinear and chaotic regime in which small modeling errors amplify rapidly. Figure~\ref{fig:lorenz} compares the mean trajectory \((m_{t,x},m_{t,y},m_{t,z})\) of the assimilated system with the reference attractor for \(\lambda\in\{1000,100,10\}\). When \(\lambda=1000\), the corrected trajectory remains close to the true attractor and the mean error stays uniformly small. For \(\lambda=100\), the dynamics still capture the correct global structure, but a visible phase discrepancy remains. When \(\lambda=10\), the correction is too weak and the bottom panel shows repeated spikes in the mean-trajectory error, corresponding to intermittent departures from the correct lobe of the attractor. These results show that density-level feedback can stabilize a biased microscopic forecast even in a chaotic setting.

\begin{figure}
    \centering
    \includegraphics[width=0.95\linewidth]{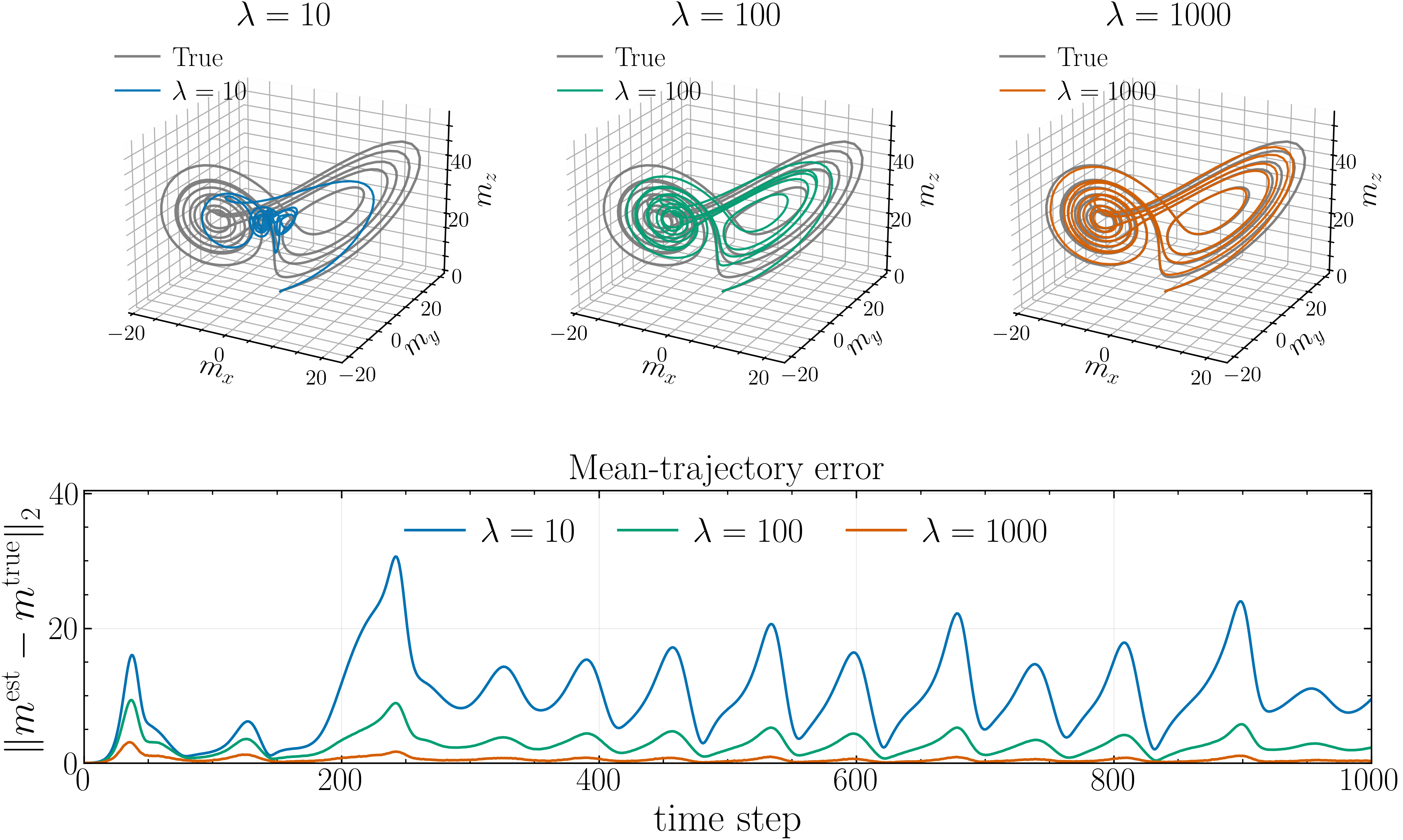}
    \caption{\textbf{Mean-field Lorenz test.} Top: trajectories of the mean state \((m_{t,x},m_{t,y},m_{t,z})\) for \(\lambda=1000,100,10\) compared with the reference attractor. Bottom: time series of the mean-trajectory error. Stronger nudging keeps the corrected dynamics close to the true attractor and suppresses intermittent excursions.}
    \label{fig:lorenz}
\end{figure}

\subsection{Vlasov--Poisson equation}
We next consider the one-dimensional Vlasov--Poisson system
\[
\begin{aligned}
    \intd X_t^i &= V_t^i \,\intd t,\\
    \intd V_t^i &= E(X_t^i,t)\,\intd t,
\end{aligned}
\]
where the self-consistent electric field is given by
\[
E(x,t) = - \nabla_x \phi(x,t),
\]
and the electrostatic potential $\phi$ solves the Poisson equation
\[
-\Delta \phi = \rho - 1, 
\qquad 
\rho(x,t) = \int f(x,v,t)\,\intd v.
\]
In these Vlasov-Poisson experiments, observations are taken on the full
phase space. Namely, for \(z=(x,v)\), we use the smoothed phase-space density
\[
    y_t(z)=(K_h*f_t)(z),
\]
and compute the nudging residual between this observation and the smoothed
forecast density.  Although \(\rho_t\) determines the electric
field through Poisson's equation, it does not uniquely determine the full
kinetic distribution \(f(x,v,t)\). The spatial-marginal observation problem is
therefore left to future work.
\paragraph{Landau damping}
We first consider the classical Landau damping problem. The ground-truth dynamics are initialized with
\[
f(x,v,0) \propto f_{M,1}(v)\bigl(1+\varepsilon \cos(\kappa x)\bigr), 
\qquad 
f_{M,\sigma}(v) = \frac{1}{\sqrt{2\pi}\sigma}\exp\left(-\frac{v^2}{2\sigma^2}\right),
\]
which corresponds to a Maxwellian velocity distribution perturbed by a spatial density modulation. In our experiments, we take $\varepsilon = 0.5$ and $\kappa = \tfrac{1}{2}$.
The spatial domain is periodic with period $4\pi$, i.e., $x \in [0,4\pi]$, while the velocity domain is truncated to $v \in [-6,6]$ for numerical implementation.

To quantify the damping, we monitor the amplitude of the fundamental
Fourier mode of the electric field,
\[
E_k(t) = \left|\hat{E}(k,t)\right|,
\qquad
\hat{E}(k,t) = \frac{1}{L}\int_0^{L} E(x,t)\,e^{-ikx}\,\intd x,
\]
with $k=\kappa$ and $L=4\pi$. In the linear regime, $E_k$ decays
exponentially at the Landau damping rate.

The forecast model uses the same dynamical system but is initialized
from an unperturbed Maxwellian ($\varepsilon=0$), resulting in a biased
prediction that lacks the initial density modulation.
Figure~\ref{fig:landau_damping} compares the reference solution, the
biased forecast, and the assimilated dynamics. The results show that the
nudging term successfully reconstructs the damping behavior and recovers
the correct exponential decay of~$E_k$.
\begin{figure}
    \centering
     \includegraphics[width=0.5\linewidth]{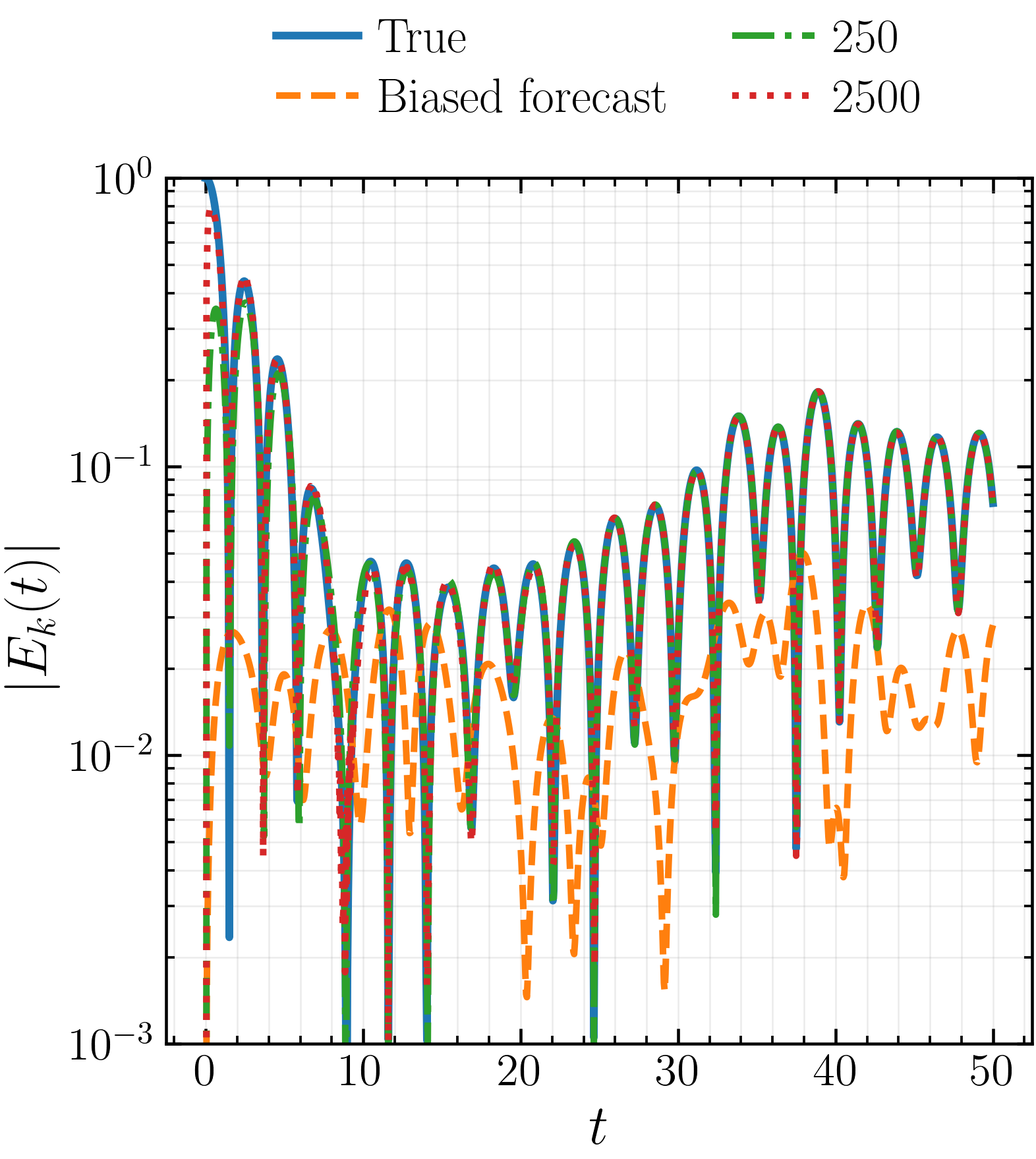}
    \caption{\textbf{Landau damping.} Time evolution of the fundamental electric-field mode 
$|E_k(t)|$ for the reference dynamics, the biased forecast, and the assimilated forecasts.
The nudged systems recover the damping profile from density-level observations.}
    \label{fig:landau_damping}
\end{figure}

\paragraph{Two-stream instability}
We also consider the two-stream instability, which exhibits strongly nonlinear behavior. The ground-truth initial condition is given by
\[
f(x,v,0) \propto \bigl(0.5 f_{M,\sigma}(v-u_0) + 0.5 f_{M,\sigma}(v+u_0)\bigr)\bigl(1+\varepsilon \cos(\kappa x)\bigr),
\]
with parameters $\sigma=0.2$, $u_0=1$, $\varepsilon=0.01$, and $\kappa=0.5$. This corresponds to a bimodal velocity distribution with a small spatial perturbation. In contrast, the forecast model is initialized from a single Maxwellian profile,
\[
f(x,v,0) \propto f_{M,1}(v)\bigl(1+\varepsilon \cos(\kappa x)\bigr),
\]
which fails to capture the underlying two-stream structure.

Figure~\ref{fig:two_stream} shows the evolution of the system at different times. The biased model is unable to reproduce the filamentation and phase-space structures characteristic of the instability. In contrast, the assimilated dynamics recover the correct qualitative features, demonstrating that the proposed nudging mechanism can reconstruct complex multimodal structures from macroscopic observations.

\begin{figure}
    \centering
    \includegraphics[width=0.49\linewidth]{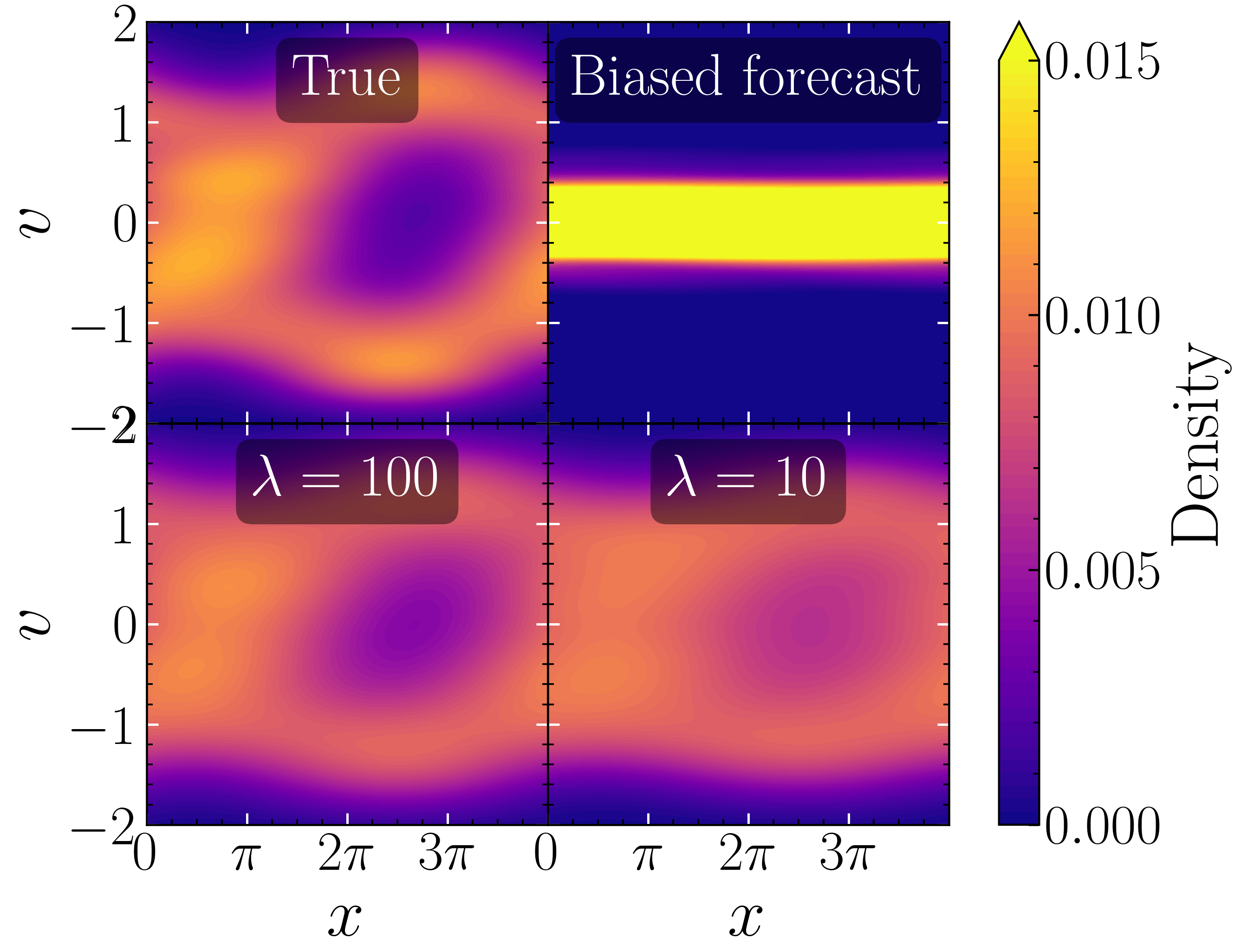}
    \includegraphics[width=0.49\linewidth]{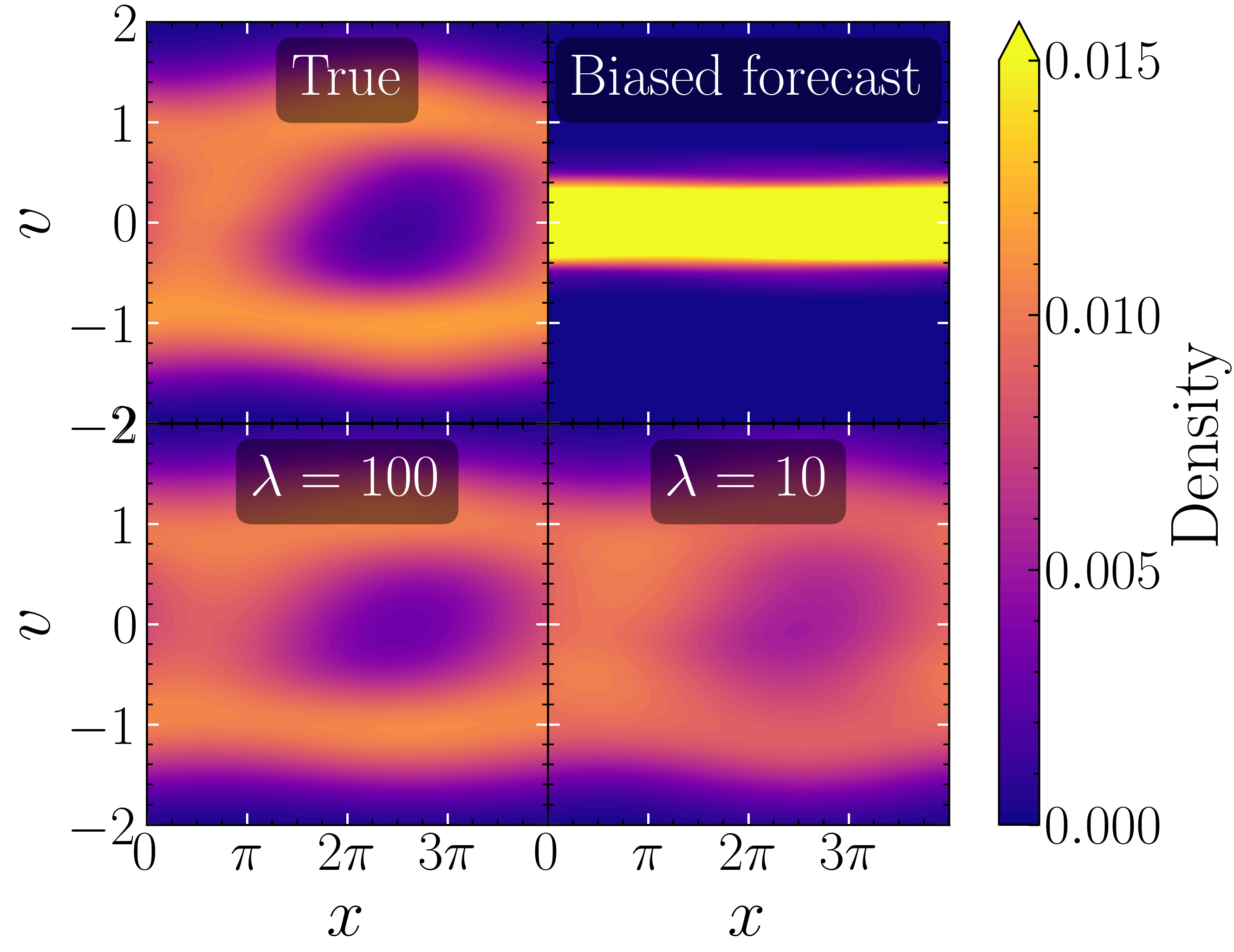}
    \caption{\textbf{Two-stream instability.} Phase-space density at
intermediate (left, step 1000) and late (right, step 2000) times. Each
panel compares the reference two-stream dynamics (top), the biased
single-Maxwellian forecast (center), and the assimilated solution
(bottom). The nudging recovers the filamentation and vortex-merging
structures absent from the biased model.}
\label{fig:two_stream}
\end{figure}

\subsection{Collective-motion data}
\label{subsec:fish}

We apply Multiscale Nudging to trajectory data obtained from
experiments with approximately $N\approx1126$ fish swimming in a
quasi-two-dimensional tank, processed by the TREX
tracker~\cite{walter2021trex}. The dataset consists of discrete-time
position measurements $\{\mathbf{X}_i(t_n)\}_{n=0}^{T}$, which are spatially
and temporally incomplete due to visual occlusions and tracking dropouts;
at every frame, a substantial fraction of trajectories carry a
missing flag, and identities are not consistent across
frames~\cite{walter2021trex}. Exact particle-to-particle correspondence
is therefore unavailable, which makes a permutation-invariant,
density-level assimilation strategy the natural choice.

Since the underlying physical interactions among the fish are not
known in closed form, we follow~\cite{lyu2026mvnn} and use a
learned mean-field drift
\begin{equation}
  \mathbf{b}_{\mathrm{model}}(\mathbf{x},\nu)
  \;=\;
  \mathrm{MLP}_{\theta_2}\bigl(\mathbf{x},\, \langle
       \mathrm{MLP}_{\theta_1}(\cdot)\rangle_{\nu}\bigr)
  \label{eq:fish_model}
\end{equation}
trained on a held-out portion of the same dataset. Two modest MLPs
($[12,24,48]$ and $[128,128,128,2]$) are composed as a feature network and a descriptor network, with the inner empirical mean providing the mean-field
coupling. The trained drift captures the average circulating motion of
the school of fish but, as we show below, accumulates errors when rolled out
without observational feedback.

The observation operator $\mathcal{O}_h$ is the same Gaussian KDE used throughout this paper; at every observed frame we evaluate the smoothed
density $\mu^{\mathrm{obs}}_t = K_h*\hat{\mu}^{\mathrm{TREX}}_t$ on a
$125\times125$ grid covering the $[0,L]^2$ tank
($L=128\,\mathrm{px}$) with bandwidth $h=2\,\mathrm{px}$ and
exclude particles flagged as missing from the empirical measure used to
form $\mu^{\mathrm{obs}}_t$. The assimilated forecast
is integrated by Euler--Maruyama with $\Delta t=0.025\,\mathrm{s}$
and $L_{\mathrm{nud}}=100$ inner nudging substeps per outer step;
the nudging strength is $\lambda=1$. We compare three trajectories
sharing the same initial configuration $\mathbf{Z}_0=\mathbf{X}(t_0)$:
the experimental reference, a biased forecast obtained by rolling
out~\eqref{eq:fish_model} without any observation feedback, and the
Multiscale Nudging forecast obtained from the smoothed-density
observations.

Figure~\ref{fig:fish_particles} compares the raw image, the reference
ensemble, the biased forecast, and the Multiscale Nudging forecast at four
snapshots $t\in\{0,\,1.25,\,3.75,\,6.25\}\,\mathrm{s}$. Each particle
is rendered as a small line segment oriented along its instantaneous
velocity, so that the ring-shaped circulation pattern of the school is
directly visible. The reference school maintains a coherent
counterclockwise circulation along the tank boundary throughout the
window; the biased forecast captures the initial pattern but loses
coherence quickly. In particular, by $t=3.75\,\mathrm{s}$, the rotational structure
is fragmented, and spurious particles concentrate near the upper-left
corner. The Multiscale Nudging forecast tracks both the boundary
circulation and the central low-density core for the full window.
The same comparison at the level of the smoothed density is shown in
Figure~\ref{fig:fish_density}. The reference field exhibits a
characteristic ring of high occupancy along the tank boundary together
with a structured interior. The biased forecast develops a spurious
concentration in the upper-left corner at $t=1.25\,\mathrm{s}$
and never recovers the boundary ring, whereas the assimilated forecast
matches the reference density throughout.

To quantify the assimilation, we compute the $L^2$ error of the density:
\[
  \mathrm{Err}_t \;=\;
  \bigl\| K_h*\hat{\mu}_t^{\mathrm{forecast}}
        - K_h*\hat{\mu}_t^{\mathrm{ref}} \bigr\|_{L^2(\Omega)},
\]
where the empirical measures use only the unmasked particles at each
frame. Figure~\ref{fig:fish_l2} reports $\mathrm{Err}_t$ for the biased
and assimilated forecasts on a logarithmic scale. The biased forecast
saturates near $\mathrm{Err}\approx30$ within roughly one second,
whereas the Multiscale Nudging forecast remains close to
$\mathrm{Err}\approx2$ for the entire window; the assimilated error
stays approximately one order of magnitude below the biased baseline.
\begin{figure}[t]
  \centering
  \includegraphics[width=0.96\linewidth]{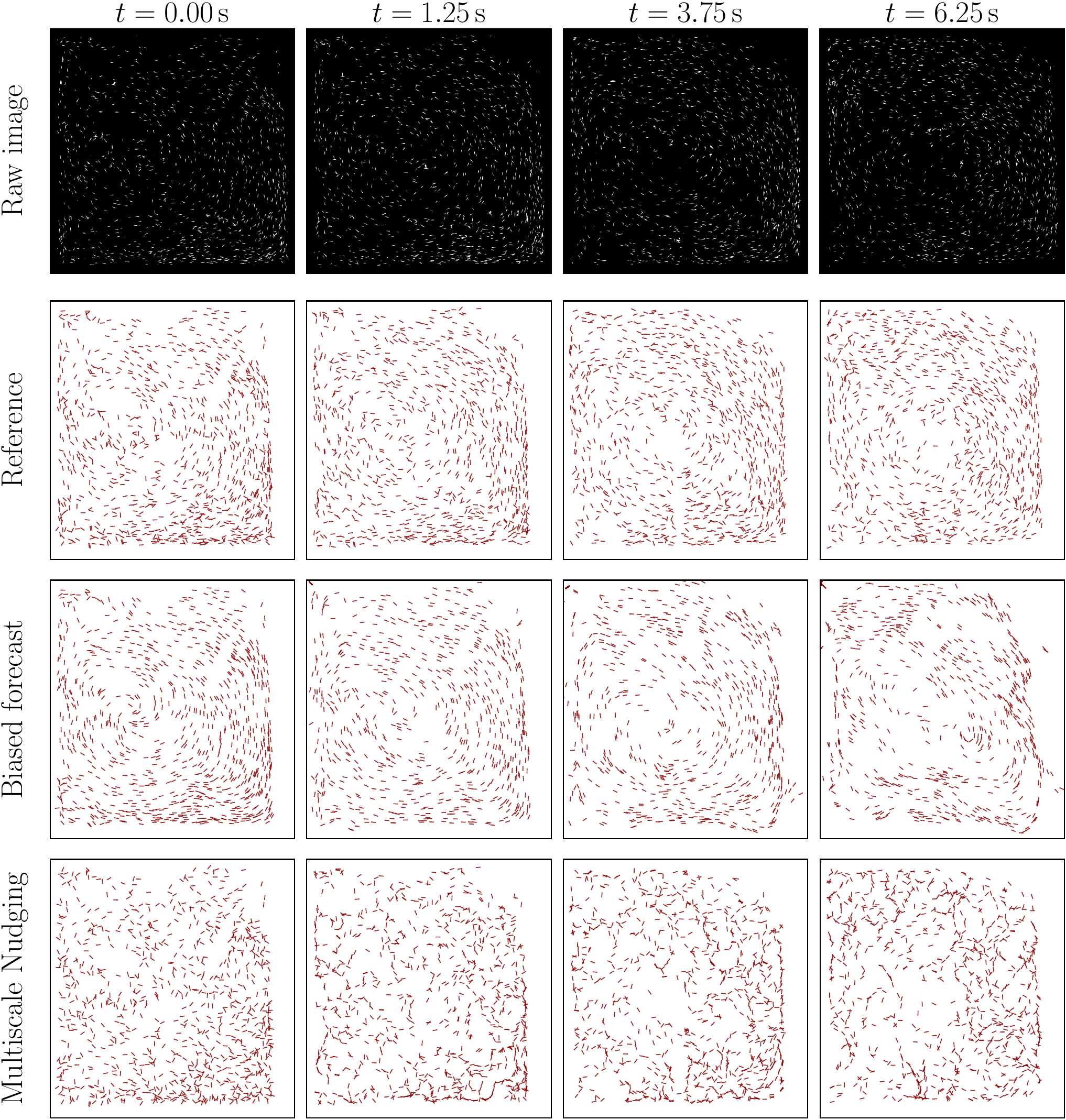}
  \caption{\textbf{Particle-level comparison on the fish dataset.}
   Rows, top to bottom: raw image frame, reference
   ensemble, biased forecast (learned drift, no observation feedback),
   and Multiscale Nudging. Columns: $t = 0,\,1.25,\,3.75,\,6.25\,$s after
   the common initial condition. Each particle is shown as a short
   segment oriented along its instantaneous velocity. The biased
   forecast loses the boundary ring and accumulates spurious mass in
   the upper-left corner; the Multiscale Nudging forecast preserves the
   circulation pattern of the school.}
  \label{fig:fish_particles}
\end{figure}

\begin{figure}[t]
  \centering
  \includegraphics[width=0.96\linewidth]{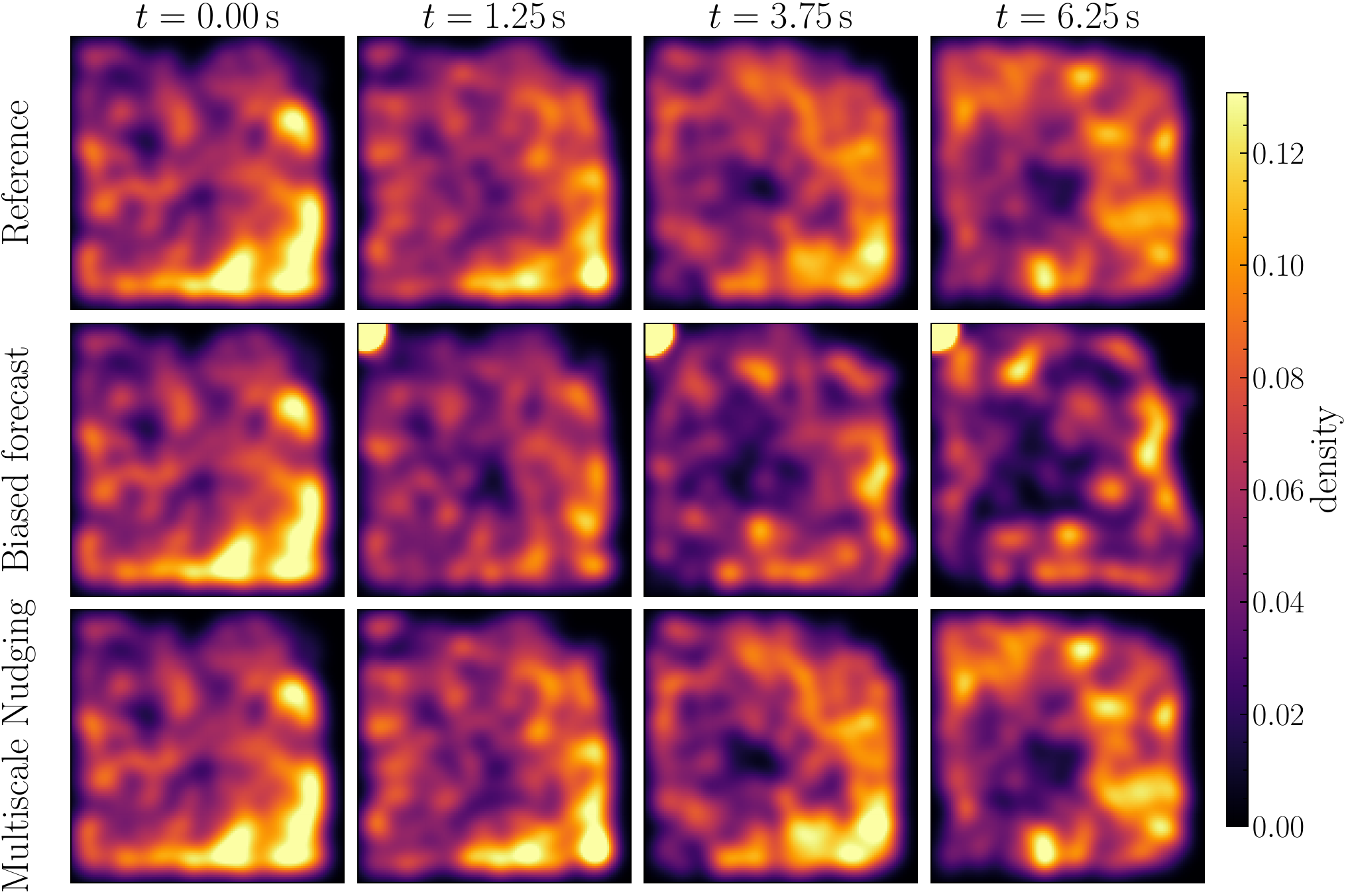}
  \caption{\textbf{Smoothed-density comparison on the fish dataset.}
   Rows: reference density, biased forecast density, Multiscale Nudging
   density. Columns: $t = 0,\,1.25,\,3.75,\,6.25\,$s. Densities are
   evaluated by Gaussian KDE with bandwidth $h=2\,$px on a
   $125\times125$ grid.
   The biased forecast develops a singular concentration in the
   upper-left corner; the assimilated forecast tracks the reference
   throughout.}
  \label{fig:fish_density}
\end{figure}

\begin{figure}[t]
  \centering
  \includegraphics[width=0.6\linewidth]{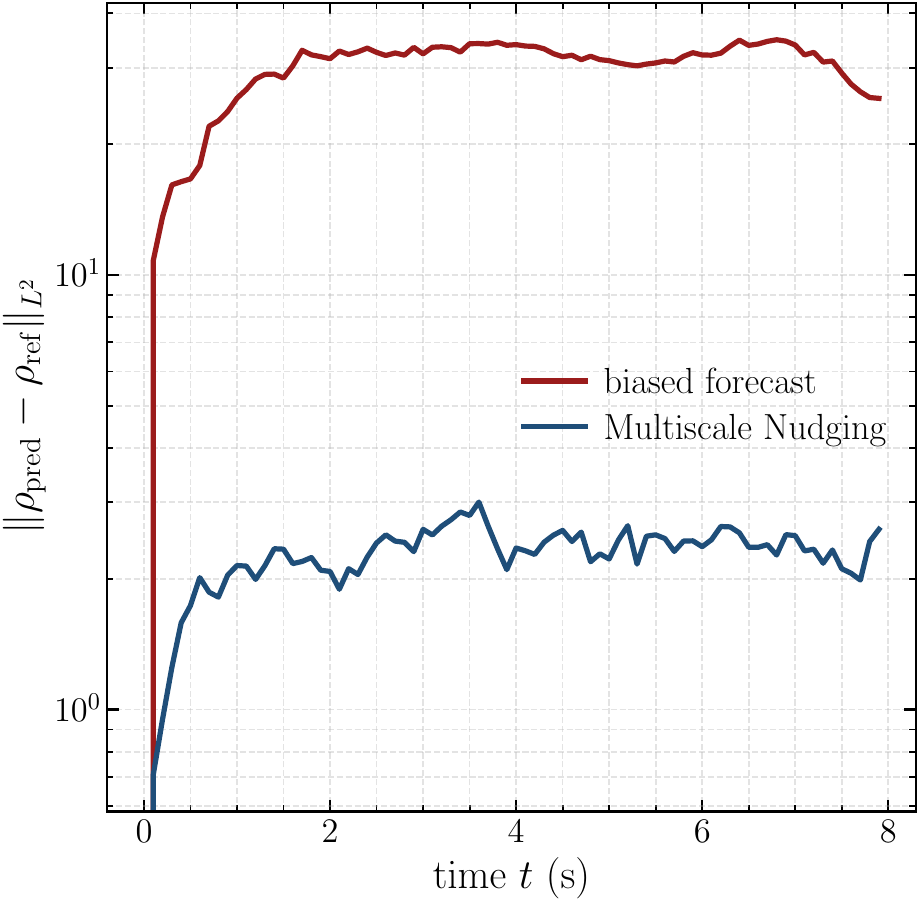}
  \caption{\textbf{Density $L^2$ error on the fish dataset.}
   Dynamics of
   $\| K_h*\hat{\mu}_t^{\mathrm{forecast}}
       - K_h*\hat{\mu}_t^{\mathrm{ref}}\|_{L^2}$ for the biased
   forecast (red) and Multiscale Nudging (blue). The biased forecast
   saturates near $30$ within $\sim 1\,$s; the assimilated forecast
   stays approximately one order of magnitude lower.}
  \label{fig:fish_l2}
\end{figure}

\section{Conclusion}
We introduced Multiscale Nudging, a measure-based method for assimilating coarse observations into microscopic mean-field particle dynamics. The method addresses a representation mismatch that arises when the forecast is a labeled particle ensemble, but the data are available only as smoothed, permutation-invariant densities. By defining the observation mismatch on probability measures and applying its Wasserstein gradient at the particle level, the method produces a practical feedback correction without particle matching, model linearization, or ensemble covariance estimation. For a fixed observation scale, we established well-posedness of the assimilated McKean-Vlasov dynamics and propagation of chaos for the particle approximation. We also proved an \(L^2\)-stability estimate under exact smoothed observations and a kernel-scale observability condition, showing exponential convergence up to a model-error-dependent bias floor. Numerical results across Gaussian, multimodal, chaotic, kinetic, and experimental collective-motion systems show that the approach can correct biased microscopic forecasts using only density-level information.

Several limitations remain. The present stability result assumes exact observations and requires the error to be observable at the kernel scale. In problems with hidden variables, such as kinetic systems observed only through spatial density, some phase-space structures may be only partially identifiable. Extending the method to noisy observations, nonlinear observation operators, adaptive bandwidths, and partial observability is a natural direction for future work.

\section*{Acknowledgments}
 This work used Anvil at Purdue through allocation MTH260007 from the Advanced Cyberinfrastructure Coordination Ecosystem: Services \& Support (ACCESS) program\cite{boerner2023access}, which is supported by U.S. National Science Foundation grants \#2138259, \#2138286, \#2138307, \#2137603, and \#2138296.
\bibliographystyle{unsrt}  
\bibliography{main}  

\appendix

\appendix
\section{Proof of Lemma \ref{lemma:kernel}}
\label{app:proof_kernel}
\begin{proof}
We first prove the regularity statements. The periodized Gaussian
\(K_h^{\mathbb T}\) is \(C^\infty\), periodic, even, and has unit mass on
\(\mathbb T^d\). Moreover, all of its derivatives are obtained by
periodizing derivatives of the Euclidean Gaussian, and the corresponding
series converge uniformly for each fixed \(h>0\). Hence
\(K_h^{\mathbb T}\in C^\infty(\mathbb T^d)\). Since convolution preserves
smoothness on the torus,
\[
\widetilde K_h=K_h^{\mathbb T}*_{\mathbb T}K_h^{\mathbb T}
\in C^\infty(\mathbb T^d).
\]
Because \(\mathbb T^d\) is compact, all derivatives of \(\widetilde K_h\)
are bounded. In particular,
\[
\nabla\widetilde K_h\in L^\infty(\mathbb T^d),
\qquad
D^2\widetilde K_h\in L^\infty(\mathbb T^d).
\]
Thus \(\widetilde K_h\) is globally Lipschitz, and
\(\nabla\widetilde K_h\) is globally Lipschitz. This proves (1) and (2).

We now prove the approximation statement. Let
\[
v(\bx)=\sum_{\bk\in\mathbb Z^d}\widehat v_{\bk}
e^{2\pi i\bk\cdot\bx}
\]
be the Fourier series of \(v\). The Fourier coefficients of the
periodized Gaussian are
\[
\widehat{K_h^{\mathbb T}}(\bk)
=
e^{-\pi^2h^2|\bk|^2},
\]
therefore,
\[
\widehat{\widetilde K_h}(\bk)
=
\widehat{K_h^{\mathbb T}}(\bk)^2
=
e^{-2\pi^2h^2|\bk|^2}.
\]
Since differentiation commutes with convolution on \(\mathbb T^d\),
\[
\reallywidehat{\left(\nabla v-\nabla(\widetilde K_h*_{\mathbb T}v)\right)}(\bk)=
2\pi i\bk
\left(1-e^{-2\pi^2h^2|\bk|^2}\right)
\widehat v_{\bk}.
\]
By Parseval's identity,
\[
\|\nabla v-\nabla(\widetilde K_h*_{\mathbb T}v)\|_{L^2}^2
=
\sum_{\bk\in\mathbb Z^d}
(2\pi)^2|\bk|^2
\left|1-e^{-2\pi^2h^2|\bk|^2}\right|^2
|\widehat v_{\bk}|^2 .
\]

If \(v\in H^1(\mathbb T^d)\), then
\[
\sum_{\bk\in\mathbb Z^d}
(2\pi)^2|\bk|^2|\widehat v_{\bk}|^2<\infty,
\]
so for each fixed \(\bk\), we have the following limit
\[
1-e^{-2\pi^2h^2|\bk|^2}\to0
\qquad\text{as }h\downarrow0.
\]
The coefficients are bounded above and below by:
\[0\le 1-e^{-2\pi^2h^2|\bk|^2}\le1.
\]
And applying the dominated convergence theorem therefore yields the limit
\[
\|\nabla v-\nabla(\widetilde K_h*_{\mathbb T}v)\|_{L^2(\mathbb T^d)}
\to0
\qquad\text{as }h\downarrow0.
\]
It remains to prove the quantitative estimate for \(v\in H^2(\mathbb T^d)\).
Using the elementary inequality
\[
1-e^{-t}\le C\sqrt t,
\qquad t\ge0,
\]
we obtain
\[
1-e^{-2\pi^2h^2|\bk|^2}
\le
Ch|\bk|.
\]
Hence
\[
\begin{aligned}
\|\nabla v-\nabla(\widetilde K_h*_{\mathbb T}v)\|_{L^2}^2
&=
\sum_{\bk\in\mathbb Z^d}
(2\pi)^2|\bk|^2
\left|1-e^{-2\pi^2h^2|\bk|^2}\right|^2
|\widehat v_{\bk}|^2  \\
&\le
Ch^2
\sum_{\bk\in\mathbb Z^d}
|\bk|^4|\widehat v_{\bk}|^2 .
\end{aligned}
\]
By the Fourier characterization of the Sobolev norm, we have
\[
\sum_{\bk\in\mathbb Z^d}
|\bk|^4|\widehat v_{\bk}|^2
\sim
\|D^2v\|_{L^2(\mathbb T^d)}^2,
\]
and thus we conclude that
\[
\|\nabla v-\nabla(\widetilde K_h*_{\mathbb T}v)\|_{L^2(\mathbb T^d)}
\le
Ch\|D^2v\|_{L^2(\mathbb T^d)},
\]
which completes the proof.
\end{proof}
\section{Proof of Proposition \ref{prop:well_possedness}}
\label{app:proof_well_possedness}
\begin{proof}
The proof is based on Theorem 1.7 in \cite{carmona2016lectures}. We only need to show that for all $\bx_1,\bx_2 \in \mathbb R^d $ and $\nu_1,\nu_2\in \mathcal P_2 (\mathbb R^d)$, 
\begin{align*}
\|&\bb_{\mathrm{model}}(\bx_1,\nu_1) - \lambda \nabla (\tilde{K}_h*\nu_1-K_h * \mu^\mathrm{obs})(\bx_1) \\
&\quad - (\bb_{\mathrm{model}}(\bx_2,\nu_2) - \lambda \nabla (\tilde{K}_h*\nu_2-K_h * \mu^\mathrm{obs})(\bx_2))\| \\ & \leq C(\|\bx_1-\bx_2\| + W_2(\nu_1,\nu_2)).
\end{align*}

By the Lipschitz property of $\bb_\mathrm{model}$ and triangle inequality, we have:
\begin{equation}
\label{equ:proof_eq}
\begin{aligned}
\|\bb_{\mathrm{model}}&(\bx_1,\nu_1) - \lambda \nabla (\tilde{K}_h*\nu_1-K_h * \mu^\mathrm{obs})(\bx_1) \\
&\quad - (\bb_{\mathrm{model}}(\bx_2,\nu_2) - \lambda \nabla (\tilde{K}_h*\nu_2-K_h * \mu^\mathrm{obs})(\bx_2))\| \\
\leq & c(\|\bx_1-\bx_2\|+ W_2(\nu_1,\nu_2)) + \lambda \|\nabla(\tilde{K}_h*\nu_1)(\bx_1)-\nabla(\tilde{K}_h*\nu_2)(\bx_1)\| + \\
&\lambda \|\nabla(\tilde{K}_h*\nu_2)(\bx_1)-\nabla(\tilde{K}_h*\nu_2)(\bx_2)\| \\
&\quad + \lambda \|\nabla(K_h*\mu^\mathrm{obs})(\bx_1)-\nabla(K_h*\mu^\mathrm{obs})(\bx_2)\|
\end{aligned}
\end{equation}
The difference is bounded by
\begin{align*}
& \|\nabla(\tilde{K}_h*\nu_1)(\bx_1)-\nabla(\tilde{K}_h*\nu_2)(\bx_1)\| \\ =& \left\|\int \nabla \tilde{K}_h(\bx_1-\by)\intd\nu_1(\by) -\int \nabla \tilde{K}_h (\bx_1-\by)\intd\nu_2(\by)\right\|\\
 =  & \left\|\int \nabla \tilde{K}_h(\bx_1-\by)\intd(\nu_1-\nu_2)(\by)\right\|\\
\leq & L W_1(\nu_1,\nu_2) \\
\leq & L W_2(\nu_1,\nu_2)
\end{align*}
using Kantorovich-Rubinstein duality and the Lipschitz continuity of $\nabla \tilde{K}_h$.

Using the global Lipschitz properties and that these measures are unit mass 
\begin{align*}
&\|\nabla(\tilde{K}_h*\nu_2)(\bx_1)-\nabla(\tilde{K}_h*\nu_2)(\bx_2)\| \\
=& \left\|\int \nabla \tilde{K}_h (\bx_1-\by)\intd\nu_2(\by)-\int\nabla\tilde{K}_h (\bx_2-\by)\intd\nu_2(\by)\right\|\\
\leq & \int \|\nabla\tilde{K}_h (\bx_1-\by)-\nabla\tilde{K}_h(\bx_2-\by)\|\intd\nu_2(\by)\\
\leq & L \|\bx_1-\bx_2\|
\end{align*}
We have that
\begin{align*}
&(K_h*\mu^\mathrm{obs})(\bx)\\
=& \int K_h(\bx-\bz)\mu^\mathrm{obs}(\bz)\intd\bz\\
=&\int \int K_h(\bx-\bz)K_h(\bz-\by)\mu_t(\intd\by)\intd\bz\\
=& (\tilde{K_h}*\mu_t)(\bx).
\end{align*}
So for the last term on the right-hand side of \eqref{equ:proof_eq}, 
\begin{align*}
&\|\nabla(K_h*\mu^\mathrm{obs})(\bx_1)-\nabla(K_h*\mu^\mathrm{obs})(\bx_2)\| \\
=& \|\nabla(\tilde{K}_h*\mu_t)(\bx_1)-\nabla(\tilde{K}_h*\mu_t)(\bx_2)\|\\
=& \left\|\int \nabla \tilde{K}_h(\bx_1-\by)\intd\mu_t(\by)-\int\nabla\tilde{K}_h(\bx_2-\by) \intd\mu_t(\by)\right\|\\
\leq& \int \|\nabla\tilde{K}_h(\bx_1-\by)-\nabla\tilde{K}_h(\bx_2-\by)\| \intd\mu_t(\by)\\
\leq& L \|\bx_1-\bx_2\|
\end{align*}
Simplifying \eqref{equ:proof_eq}, we obtain the desired result:
\begin{align*}
&\|\bb_{\mathrm{model}}(\bx_1,\nu_1) - \lambda \nabla (\tilde{K}_h*\nu_1-K_h * \mu^\mathrm{obs})(\bx_1) - \\&(\bb_{\mathrm{model}}(\bx_2,\nu_2) - \lambda \nabla (\tilde{K}_h*\nu_2-K_h * \mu^\mathrm{obs})(\bx_2))\| \\
\leq & c(\|\bx_1-\bx_2\|+ W_2(\nu_1,\nu_2)) + \lambda L W_2(\nu_1,\nu_2) +\lambda L \|\bx_1-\bx_2\| + \lambda L \|\bx_1-\bx_2\|\\
\leq & C(\|\bx_1 - \bx_2\| + W_2(\nu_1,\nu_2)),
\end{align*}
where $C:=2\lambda L + c$. Thus the modified drift inherits the required global Lipschitz continuity in both its spatial and measure arguments, which is the key condition needed for the well-posedness.
\end{proof}

\section{Proof of Proposition \ref{equ:propagation_of_chaos} }
\label{app:propagation_of_chaos}
\begin{proof}
We use a synchronous coupling argument, as in \cite{sznitman2006topics} and Theorem 3.1 of \cite{chaintron2021propagation}. We construct the $N$-particle system $(\bZ_t^{i,N})_{i=1}^N$ and a set of $N$ independent mean-field processes with the same Brownian motions and show that their $L^2$ distance vanishes as $N\to \infty$. We define $N$ independent processes $(\bar{\bZ}^{1,N}_t,\dots,\bar{\bZ}^{N,N}_t)$ defined as the solutions of $N$ SDEs:
\begin{equation*}
\intd \bar{\bZ}^{i,N}_t
= \bb_{\mathrm{model}}\left(\bar{\bZ}^{i,N}_t, f_t \right)\,\intd t-  \lambda\left(\nabla \tilde{K}_h*f_t(\bar{\bZ}^{i,N}_t)  -  \nabla   K_h *\mu^{\mathrm{obs}} (\bar{\bZ}^{i,N}_t)\right)\,\intd t  + \Sigma \,\intd \bW^i_t,
\end{equation*}
for $i\in\{1,\dots,N\}$, where $(\bW^i_t)$ is the same Brownian motion as in \eqref{equ:density_obs} and $f_t = \mathrm{Law}(\bar{\bZ}^{i,N}_t)$. We will show that:
\begin{equation}\label{equ:bound}
   \frac{1}{N}\sum_{i=1}^N \mathbb E\left[\sup_{t\leq T}\left|\bZ^{i,N}_t-\bar{\bZ}^{i,N}_t\right|^2\right] \leq \epsilon (N,T). 
\end{equation} 
We fix $i=1$ and define the path-space coupling
\[
\pi_N:=\mathrm{Law}\!\left( (\bZ_t^{1,N})_{t\in[0,T]},\,(\bar \bZ_t^{1,N})_{t\in[0,T]} \right).
\]
Its marginals are $f_{[0,T]}^{1,N}$ and $f_{[0,T]}$, respectively. By the definition of the 2-Wasserstein distance on path-space,
\[
\begin{aligned}
W_2^2\!\left(f_{[0,T]}^{1,N},f_{[0,T]}\right)
&\le
\int \sup_{0\le t\le T}\left|\bx_t-\by_t\right|^2\,\pi_N(\intd \bx,\intd \by) \\
&=
\mathbb{E}\!\left[\sup_{0\le t\le T}\left|\bZ_t^{1,N}-\bar \bZ_t^{1,N}\right|^2\right].
\end{aligned}
\]
Since the right-hand side of Equation \eqref{equ:bound} will tend to $0$ as $N\to\infty$, we can conclude that
\[
\lim_{N\to\infty} W_2\left(f_{[0,T]}^{1,N},f_{[0,T]}\right) = 0.
\]
Furthermore, by exchangeability, the same estimate holds for any fixed finite collection of particles, which yields $f_{[0,T]}$-chaoticity of the nudged particle system. 

By Ito's formula, and since the stochastic terms cancel due to the synchronous coupling, we have:
\begin{equation*}
\begin{aligned}
\left|\bZ^{i,N}_t - \bar{\bZ}^{i,N}_t\right|^2 = 
2 \int_0^t &\bigg\langle \bZ^{i,N}_s - \bar{\bZ}^{i,N}_s, 
\bb_\mathrm{model}(\bZ^{i,N}_s, \nu_s^N) - \bb_\mathrm{model}(\bar{\bZ}^{i,N}_s, f_s) - \\
&\lambda \Big(\frac{1}{N}\sum_{j=1}^N \nabla \tilde  K_h(\bZ^{i,N}_s - \bZ^{j,N}_s) - \nabla\tilde{K}_h * f_t(\bar{\bZ}^{i,N}_s)  \Big) +\\
& \lambda\Big(\nabla   K_h *\mu^{\mathrm{obs}} (\bZ^{i,N}_s) - \nabla   K_h *\mu^{\mathrm{obs}} (\bar{\bZ}^{i,N}_s)\Big)\bigg\rangle \intd s.
\end{aligned}
\end{equation*}
We take the supremum and then the expectation:
\begin{equation}
\label{equ:prop_chaos}
\begin{aligned}
\mathbb E&\left[\sup_{t\leq T}\left|\bZ^{i,N}_t-\bar{\bZ}^{i,N}_t\right|^2\right] \\
&= \mathbb E\Bigg[\sup_{t\leq T} \bigg| 2 \int_0^t \bigg\langle \bZ^{i,N}_s-\bar{\bZ}^{i,N}_s, \bb_\mathrm{model}(\bZ^{i,N}_s,\nu_s^N) - \bb_\mathrm{model}(\bar{\bZ}^{i,N}_s, f_s) - \\
& \lambda \Big(\frac{1}{N}\sum_{j=1}^N \nabla \tilde  K_h(\bZ^{i,N}_s - \bZ^{j,N}_s) - \nabla\tilde{K}_h*f_t(\bar{\bZ}^{i,N}_s)  \Big) \\
&\quad + \lambda\Big(\nabla   K_h *\mu^{\mathrm{obs}} (\bZ^{i,N}_s) - \nabla   K_h *\mu^{\mathrm{obs}} (\bar{\bZ}^{i,N}_s)\Big)\bigg\rangle \intd s \bigg| \Bigg] \\
\leq & 2 \int^T_0 \mathbb E \Bigg[ \bigg| \bigg\langle \bZ^{i,N}_s-\bar{\bZ}^{i,N}_s, \bb_\mathrm{model}(\bZ^{i,N}_s,\nu_s^N) - \bb_\mathrm{model}(\bar{\bZ}^{i,N}_s, f_s) -\\ & \lambda \Big(\frac{1}{N}\sum_{j=1}^N \nabla \tilde  K_h(\bZ^{i,N}_s - \bZ^{j,N}_s) - \nabla \tilde{K}_h *f_t(\bar{\bZ}^{i,N}_s)  \Big) \\
&\quad +\lambda\Big(\nabla   K_h *\mu^{\mathrm{obs}} (\bZ^{i,N}_s) - \nabla   K_h *\mu^{\mathrm{obs}} (\bar{\bZ}^{i,N}_s)\Big)\bigg\rangle\bigg|\Bigg] \intd s\\
\leq &\int^T_0 \mathbb E \left[\left|\bZ^{i,N}_s-\bar{\bZ}^{i,N}_s\right|^2\right] \intd s + \int^T_0 \mathbb E \Bigg[\bigg|\bb_\mathrm{model}(\bZ^{i,N}_s,\nu_s^N) - \bb_\mathrm{model}(\bar{\bZ}^{i,N}_s, f_s) -
\\&\lambda \Big(\frac{1}{N}\sum_{j=1}^N \nabla \tilde  K_h(\bZ^{i,N}_s - \bZ^{j,N}_s) - \nabla\tilde{K}_h *f_t(\bar{\bZ}^{i,N}_s) \Big) \\
&\quad +\lambda\Big(\nabla   K_h *\mu^{\mathrm{obs}} (\bZ^{i,N}_s) - \nabla   K_h *\mu^{\mathrm{obs}} (\bar{\bZ}^{i,N}_s)\Big)\bigg|^2\Bigg] \intd s \\
\leq &\int^T_0 \mathbb E \left[\sup_{r\leq s} \left|\bZ^{i,N}_r-\bar{\bZ}^{i,N}_r\right|^2\right] \intd s + \int^T_0 \mathbb E \Bigg[\bigg|\bb_\mathrm{model}(\bZ^{i,N}_s,\nu_s^N) - \bb_\mathrm{model}(\bar{\bZ}^{i,N}_s, f_s) -
\\&\lambda \Big(\frac{1}{N}\sum_{j=1}^N \nabla \tilde  K_h(\bZ^{i,N}_s - \bZ^{j,N}_s) - \nabla\tilde{K}_h*f_t(\bar{\bZ}^{i,N}_s)  \Big) \\
&\quad +\lambda\Big(\nabla   K_h *\mu^{\mathrm{obs}} (\bZ^{i,N}_s) - \nabla   K_h *\mu^{\mathrm{obs}} (\bar{\bZ}^{i,N}_s)\Big)\bigg|^2\Bigg] \intd s.
\end{aligned}
\end{equation}

In the second integral, we have that:
\begin{equation}
\label{equ:second_itegrand}
\begin{aligned}
\mathbb E &\Bigg[\bigg|\bb_\mathrm{model}(\bZ^{i,N}_s,\nu_s^N) - \bb_\mathrm{model}(\bar{\bZ}^{i,N}_s, f_s) \\
&\quad - \lambda \Big(\frac{1}{N}\sum_{j=1}^N \nabla \tilde  K_h(\bZ^{i,N}_s - \bZ^{j,N}_s) - \nabla\tilde{K}_h*f_t(\bar{\bZ}^{i,N}_s)  \Big) +\\&\lambda\Big(\nabla   K_h *\mu^{\mathrm{obs}} (\bZ^{i,N}_s) - \nabla   K_h *\mu^{\mathrm{obs}} (\bar{\bZ}^{i,N}_s)\Big)\bigg|^2\Bigg] \\
\leq & 3 \mathbb E \left[\left|\bb_\mathrm{model}(\bZ^{i,N}_s,\nu_s^N) - \bb_\mathrm{model}(\bar{\bZ}^{i,N}_s, f_s)\right|^2\right] \\
&\quad + 3 \mathbb E \left[\left|\lambda \Big(\frac{1}{N}\sum_{j=1}^N \nabla \tilde  K_h(\bZ^{i,N}_s - \bZ^{j,N}_s) - \nabla\tilde{K}_h*f_t(\bar{\bZ}^{i,N}_s)  \Big) \right|^2\right] + \\&
3\mathbb E \left[\left|\lambda\Big(\nabla   K_h *\mu^{\mathrm{obs}} (\bZ^{i,N}_s) - \nabla   K_h *\mu^{\mathrm{obs}} (\bar{\bZ}^{i,N}_s)\Big)\right|^2\right].
\end{aligned}
\end{equation}
For the first term on the right-hand side of \eqref{equ:second_itegrand}, we define $\bar{\nu}^N_s = \frac{1}{N}\sum_{j=1}^N \delta_{\bar{\bZ}^{j,N}_s}$ as the empirical measure of the system $(\bar{\bZ}^{1,N}_s,\dots,\bar{\bZ}^{N,N}_s)$ and use triangle inequality and the Lipschitz condition of $\bb_\textrm{model}$ to get:
The second term on the right-hand side of \eqref{equ:second_term} becomes:
We show that $\mathbb E \left[\left(W_2(\bar{\nu}^N_s,f_s)\right)^2\right] \to 0$. By Theorem 3 in \cite{varadarajan1958convergence}, the empirical measure $\bar{\nu}^N_s$ converges weakly to $f_s$. Furthermore, by Theorem 1.7 in \cite{carmona2016lectures}, $f_s$ has bounded second moment. We have convergence of the second moment by the strong law of large numbers. Hence, by Theorem 6.9 in \cite{villani2008optimal}, we have that $W_2(\bar{\nu}^N_s,f_s) \to 0$. If $\left(W_2(\bar{\nu}^N_s,f_s)\right)^2$ is also uniformly integrable, then by Vitali Convergence Theorem, $\mathbb E \left[\left(W_2(\bar{\nu}^N_s,f_s)\right)^2\right] \to 0$. $\left(W_2(\bar{\nu}^N_s,f_s)\right)^2$ is uniformly integrable if $\sup\limits_N\mathbb{E} \left[\left(W_2(\bar{\nu}^N_s,f_s)\right)^2 \mathbbm{1} _{\left(W_2(\bar{\nu}^N_s,f_s)\right)^2\geq R}\right] \to 0$  as $R\to \infty$. We have  
\begin{equation*}
\begin{aligned}
\left(W_2(\bar{\nu}^N_s,f_s)\right)^2 &= \inf_{\gamma\in\Gamma(\bar{\nu}^N_s, f_s)} \int |\bx-\by|^2 \intd \gamma(\bx,\by)\\
&\leq 2 \int |\bx|^2 \intd \bar{\nu}^N_s + 2 \int |\by|^2 \intd f_s,
\end{aligned}
\end{equation*}
so that
\begin{equation*}
\begin{aligned}
&\sup\limits_N\mathbb E \left[\left(W_2(\bar{\nu}^N_s,f_s)\right)^2 \mathbbm{1}_{\left(W_2(\bar{\nu}^N_s,f_s)\right)^2\geq R}\right] \\
\leq & 2 \sup\limits_N\int |\bx|^2 \intd \bar{\nu}^N_s \mathbb E \left[ \mathbbm{1}_{\left(W_2(\bar{\nu}^N_s,f_s)\right)^2\geq R}\right] + 2 \int |\by|^2 \intd f_s \sup\limits_N\mathbb E \left[\mathbbm{1}_{\left(W_2(\bar{\nu}^N_s,f_s)\right)^2\geq R}\right].
\end{aligned}
\end{equation*}
Since $\bar{\nu}^N_s$ and $f_s$ both have bounded second moments and $\sup\limits_N\mathbb E \left[\mathbbm{1}_{\left(W_2(\bar{\nu}^N_s,f_s)\right)^2\geq R}\right] \to 0$ because $W_2(\bar{\nu}^N_s , f_s) \to 0$, we have that $\sup\limits_N\mathbb E \left[\left(W_2(\bar{\nu}^N_s,f_s)\right)^2 \mathbbm{1}_{\left(W_2(\bar{\nu}^N_s,f_s)\right)^2\geq R}\right] \to 0$ as $R\to \infty$. 

Hence, we can conclude by Vitali Convergence Theorem that $\mathbb E \left[\left(W_2(\bar{\nu}^N_s,f_s)\right)^2\right] \to 0$.

For the second term on the right-hand side of \eqref{equ:second_itegrand}, we use triangle inequality to get:
\begin{equation}
\label{equ:second_term}
\begin{aligned}
& \mathbb E \left[\left|\lambda \Big(\frac{1}{N}\sum_{j=1}^N \nabla \tilde  K_h(\bZ^{i,N}_s - \bZ^{j,N}_s) - \nabla\tilde{K}_h*f_t(\bar{\bZ}^{i,N}_s)  \Big) \right|^2\right] \\
\leq& \mathbb E \left[\left|\lambda \Big(\frac{1}{N}\sum_{j=1}^N \nabla \tilde  K_h(\bZ^{i,N}_s - \bZ^{j,N}_s) - \frac{1}{N}\sum_{j=1}^N \nabla \tilde  K_h(\bar{\bZ}^{i,N}_s - \bar{\bZ}^{j,N}_s)  \Big) \right|^2\right] + \\&\mathbb E \left[\left|\lambda \Big(\frac{1}{N}\sum_{j=1}^N \nabla \tilde  K_h(\bar{\bZ}^{i,N}_s - \bar{\bZ}^{j,N}_s) - \nabla\tilde{K}_h*f_t(\bar{\bZ}^{i,N}_s)  \Big) \right|^2\right]
\end{aligned}
\end{equation}
The first term on the right-hand side of \eqref{equ:second_term} is simplified using triangle inequality and the Lipschitz property of $\nabla \tilde{K}_h$:
\begin{equation*}
\begin{aligned}
&\mathbb E \left[\left|\lambda \Big(\frac{1}{N}\sum_{j=1}^N \nabla \tilde  K_h(\bZ^{i,N}_s - \bZ^{j,N}_s) - \frac{1}{N}\sum_{j=1}^N \nabla \tilde  K_h(\bar{\bZ}^{i,N}_s - \bar{\bZ}^{j,N}_s)  \Big) \right|^2\right] \\
\leq & 2\lambda^2 \mathbb E \left[\left|\frac{1}{N}\sum_{j=1}^N \Big(\nabla \tilde  K_h(\bZ^{i,N}_s - \bZ^{j,N}_s) - \nabla \tilde  K_h(\bar{\bZ}^{i,N}_s - \bZ^{j,N}_s)  \Big) \right|^2\right] + \\&2\lambda^2 \mathbb E \left[\left|\frac{1}{N}\sum_{j=1}^N \Big(\nabla \tilde  K_h(\bar{\bZ}^{i,N}_s - \bZ^{j,N}_s) - \nabla \tilde  K_h(\bar{\bZ}^{i,N}_s - \bar{\bZ}^{j,N}_s)  \Big) \right|^2\right]\\
\leq & 2\lambda^2 \mathbb E \left[\frac{1}{N}\sum_{j=1}^N \left|\nabla \tilde  K_h(\bZ^{i,N}_s - \bZ^{j,N}_s) - \nabla \tilde  K_h(\bar{\bZ}^{i,N}_s - \bZ^{j,N}_s) \right|^2\right] + \\&2\lambda^2 \mathbb E \left[\frac{1}{N}\sum_{j=1}^N \left|\nabla \tilde  K_h(\bar{\bZ}^{i,N}_s - \bZ^{j,N}_s) - \nabla \tilde  K_h(\bar{\bZ}^{i,N}_s - \bar{\bZ}^{j,N}_s) \right|^2 \right]\\
\leq & 2\lambda^2 \mathbb E \left[\frac{1}{N}\sum_{j=1}^N L^2\left|\bZ^{i,N}_s - \bar{\bZ}^{i,N}_s\right|^2 \right] + 2\lambda^2 \mathbb E \left[\frac{1}{N} \sum_{j=1}^N L^2 \left|\bZ^{j,N}_s -\bar{\bZ}^{j,N}_t\right|^2\right]\\
\leq & 4\lambda^2 L^2 \mathbb E \left[\left|\bZ^{i,N}_s - \bar{\bZ}^{i,N}_s\right|^2\right].
\end{aligned}
\end{equation*}
The second term on the right-hand side of \eqref{equ:second_term} becomes:
\begin{equation*}
\begin{aligned}
&\mathbb E \left[\left|\lambda \Big(\frac{1}{N}\sum_{j=1}^N \nabla \tilde  K_h(\bar{\bZ}^{i,N}_s - \bar{\bZ}^{j,N}_s) - \nabla\tilde{K}_h*f_t(\bar{\bZ}^{i,N}_s)  \Big) \right|^2\right] \\= &\frac{\lambda^2}{N^2}\sum_{k,l=1}^N
\mathbb E \Biggl[\Big( \nabla \tilde  K_h(\bar{\bZ}^{i,N}_s - \bar{\bZ}^{k,N}_s) - \nabla\tilde{K}_h*f_t(\bar{\bZ}^{i,N}_s)  \Big) \\
&\quad \Big( \nabla \tilde  K_h(\bar{\bZ}^{i,N}_s - \bar{\bZ}^{l,N}_s) - \nabla\tilde{K}_h*f_t(\bar{\bZ}^{i,N}_s)  \Big)\Biggr]\\
\leq & \frac{\lambda^2}{N^2}\sum_{k\neq l}\mathbb E \Biggl[\Big( \nabla \tilde  K_h(\bar{\bZ}^{i,N}_s - \bar{\bZ}^{k,N}_s) - \nabla\tilde{K}_h*f_t(\bar{\bZ}^{i,N}_s)  \Big) \\
&\quad \Big( \nabla \tilde  K_h(\bar{\bZ}^{i,N}_s - \bar{\bZ}^{l,N}_s) - \nabla\tilde{K}_h*f_t(\bar{\bZ}^{i,N}_s)  \Big)\Biggr] + 4\frac{\lambda^2}{N} \|\nabla\tilde{K}_h\|_\infty^2\\
=& \frac{\lambda^2}{N^2}\sum_{k\neq l}\mathbb E \Biggl[\mathbb E \Biggl[\Big( \nabla \tilde  K_h(\bar{\bZ}^{i,N}_s - \bar{\bZ}^{k,N}_s) - \nabla\tilde{K}_h*f_t(\bar{\bZ}^{i,N}_s)  \Big) \\
&\quad \Big( \nabla \tilde  K_h(\bar{\bZ}^{i,N}_s - \bar{\bZ}^{l,N}_s) - \nabla\tilde{K}_h*f_t(\bar{\bZ}^{i,N}_s)  \Big)\Big| \bar{\bZ}^{i,N}_s\Biggr]\Biggr] + \\&4\frac{\lambda^2}{N} \|\nabla\tilde{K}_h\|_\infty^2\\
=& \frac{\lambda^2}{N^2}\sum_{k\neq l}\mathbb E \Biggl[\mathbb E \left[ \Big(\nabla \tilde  K_h(\bar{\bZ}^{i,N}_s - \bar{\bZ}^{k,N}_s) - \nabla\tilde{K}_h*f_t(\bar{\bZ}^{i,N}_s) \Big)\Big| \bar{\bZ}^{i,N}_s\right] \\
&\quad \mathbb E \left[ \Big(\nabla \tilde  K_h(\bar{\bZ}^{i,N}_s - \bar{\bZ}^{l,N}_s) - \nabla\tilde{K}_h*f_t(\bar{\bZ}^{i,N}_s) \Big)\Big| \bar{\bZ}^{i,N}_s\right]\Biggr] +\\
& 4\frac{\lambda^2}{N} \|\nabla\tilde{K}_h\|_\infty^2\\
= &4\frac{\lambda^2}{N} \|\nabla\tilde{K}_h\|_\infty^2,
\end{aligned}
\end{equation*}
where we used the law of total expectation and the fact that $\bar{\bZ}^{k,N}_s, \bar{\bZ}^{l,N}_s,$ are independent of $\bar{\bZ}^{i,N}_s$. The last equality is obtained by observing that at least one of $k,l$ is not equal to $i$; without loss of generality, assume $k\neq i$. Then since $\mathrm{Law}(\bar{\bZ}^{k,N}_s) = f_t$,
\begin{equation*}
\begin{aligned}
&\mathbb E \Bigl[ \Big(\nabla \tilde  K_h(\bar{\bZ}^{i,N}_s - \bar{\bZ}^{k,N}_s) \\
&\qquad - \nabla\tilde{K}_h*f_t(\bar{\bZ}^{i,N}_s) \Big)\Big| \bar{\bZ}^{i,N}_s = \bz\Bigr] \\
&= \mathbb E \left[\nabla\tilde{K}_h(\bz-\bar{\bZ}^{k,N}_s)\right] - \nabla \tilde{K}_h*f_t(\bz)\\
&= \int\nabla\tilde{K}_h(\bz-\by)\intd f_t(\by) - \int \nabla \tilde{K}_h(\bz-\by)\intd f_t(\by)
\\&=0
\end{aligned}
\end{equation*}
Hence, Equation \eqref{equ:second_term} simplifies to:
\begin{equation*}
\begin{aligned}
&\mathbb E \left[\left|\lambda \Big(\frac{1}{N}\sum_{j=1}^N \nabla \tilde  K_h(\bZ^{i,N}_s - \bZ^{j,N}_s) - \nabla\tilde{K}_h*f_t(\bar{\bZ}^{i,N}_s)  \Big) \right|^2\right] \\
&\quad \leq 4\lambda^2 L^2 \mathbb E \left[\left|\bZ^{i,N}_s - \bar{\bZ}^{i,N}_s\right|^2\right] + 4\frac{\lambda^2}{N} \|\nabla\tilde{K}_h\|_\infty^2.
\end{aligned}
\end{equation*}
For the third term on the right-hand side of \eqref{equ:second_itegrand}, 
\begin{equation*}
\begin{aligned}
&\mathbb E \Bigl[\Bigl|\lambda\Big(\nabla   K_h *\mu^{\mathrm{obs}} (\bZ^{i,N}_s) \\
&\qquad - \nabla   K_h *\mu^{\mathrm{obs}} (\bar{\bZ}^{i,N}_s)\Big)\Bigr|^2\Bigr] \\
&= \lambda^2\mathbb E \Bigl[\Bigl|\Big(\nabla \tilde{K}_h *\mu_t (\bZ^{i,N}_s) \\
&\qquad - \nabla \tilde{K}_h *\mu_t (\bar{\bZ}^{i,N}_s)\Big)\Bigr|^2\Bigr]\\
& = \lambda^2 \mathbb E \Bigl[\Bigl|\int \nabla \tilde{K}_h(\bZ^{i,N}_s-\by) \intd \mu_t(\by) \\
&\qquad - \int \nabla \tilde{K}_h (\bar{\bZ}^{i,N}_s - \by) \intd \mu_t(\by)\Bigr|^2\Bigr]\\
&\leq \lambda^2 \mathbb E \Bigl[\int \Bigl|\nabla \tilde{K}_h (\bZ^{i,N}_s-\by) \\
&\qquad - \nabla \tilde{K}_h (\bar{\bZ}^{i,N}_s - \by)\Bigr|^2 \intd \mu_t (\by)\Bigr] \\
&\leq \lambda^2 L^2 \mathbb E \left[ \left|\bZ^{i,N}_s - \bar{\bZ}^{i,N}_s\right|^2\right]
\end{aligned}
\end{equation*}
by the Lipschitz continuity of $\nabla \tilde{K}_h$.

Therefore, simplifying Equation \eqref{equ:second_itegrand}, we get:
\begin{equation*}
\begin{aligned}
\mathbb E &\Bigg[\bigg|\bb_\mathrm{model}(\bZ^{i,N}_s,\nu_s^N) - \bb_\mathrm{model}(\bar{\bZ}^{i,N}_s, f_s) \\
&\quad - \lambda \Big(\frac{1}{N}\sum_{j=1}^N \nabla \tilde  K_h(\bZ^{i,N}_s - \bZ^{j,N}_s) - \nabla \tilde{K}_h*f_t(\bar{\bZ}^{i,N}_s)  \Big) +\\&\lambda\Big(\nabla   K_h *\mu^{\mathrm{obs}} (\bZ^{i,N}_s) - \nabla   K_h *\mu^{\mathrm{obs}} (\bar{\bZ}^{i,N}_s)\Big)\bigg|^2\Bigg] \\
\leq & 24c^2 \mathbb E\left[\left|\bZ^{i,N}_s-\bar{\bZ}^{i,N}_s\right|^2\right] + 6c^2 \mathbb E \left[\left(W_2(\bar{\nu}^N_s,f_s)\right)^2\right] + 12\lambda^2 L^2 \mathbb E \left[\left|\bZ^{i,N}_s - \bar{\bZ}^{i,N}_s\right|^2\right] \\
&+ 12\frac{\lambda^2}{N} \|\nabla\tilde{K}_h\|_\infty^2 +
3 \lambda^2 L^2 \mathbb E \left[ \left|\bZ^{i,N}_s - \bar{\bZ}^{i,N}_s\right|^2\right]\\
\leq & (24c^2+ 15\lambda^2 L^2) \mathbb E \left[ \left|\bZ^{i,N}_s - \bar{\bZ}^{i,N}_s\right|^2\right] + 6c^2\mathbb E \left[\left(W_2(\bar{\nu}^N_s,f_s)\right)^2\right] + 12\frac{\lambda^2}{N} \|\nabla\tilde{K}_h\|_\infty^2.
\end{aligned}
\end{equation*}
Let $Y(t) = \mathbb E\left[\sup_{r\leq t}\left|\bZ^{i,N}_r-\bar{\bZ}^{i,N}_r\right|^2\right]$. Then Equation \eqref{equ:prop_chaos} becomes:
\begin{equation*}
\begin{aligned}
Y(T) &\leq \int^T_0 Y(s) \intd s + \int^T_0 (24c^2+ 15\lambda^2 L^2) \mathbb E \left[ \left|\bZ^{i,N}_s - \bar{\bZ}^{i,N}_s\right|^2\right] \\
&\quad + 6c^2\mathbb E \left[\left(W_2(\bar{\nu}^N_s,f_s)\right)^2\right]+ 12\frac{\lambda^2}{N} \|\nabla\tilde{K}_h\|_\infty^2\intd s\\
&= (1+24c^2+ 15\lambda^2 L^2) \int^T_0 Y(s) \intd s \\
&\quad + 6Tc^2\mathbb E \left[\left(W_2(\bar{\nu}^N_s,f_s)\right)^2\right] + 12\frac{\lambda^2}{N} \|\nabla\tilde{K}_h\|_\infty^2.
\end{aligned}
\end{equation*}
Thus,
\[
\begin{aligned}
    \mathbb E\left[\sup_{t\leq T}\left|\bZ^{i,N}_t-\bar{\bZ}^{i,N}_t\right|^2\right] \\
    &\leq (1+24c^2+ 15\lambda^2 L^2)\int_0^T \mathbb E\left[\sup_{t\leq s}\left|\bZ^{i,N}_t-\bar{\bZ}^{i,N}_t\right|^2\right]
    \intd s \\
    &\quad + 6Tc^2\mathbb E \left[\left(W_2(\bar{\nu}^N_s,f_s)\right)^2\right] + 12\frac{\lambda^2}{N} \|\nabla\tilde{K}_h\|_\infty^2.
\end{aligned}
\]
The conclusion follows by Gronwall's lemma and the fact that $\mathbb E \left[\left(W_2(\bar{\nu}^N_s,f_s)\right)^2\right] \to 0$.
\end{proof}
\section{Proof of Proposition~\ref{prop:positivity}}
\label{app:proof_positivity}
\begin{proof}
By Assumption~\ref{ass:regularity},
\[
\|\nabla\cdot \bb_{\mathrm{model}}(\cdot,\nu_t)\|_{L^\infty(\Omega)}
\le B_{\mathrm{model}},
\qquad \forall t\in[0,T].
\]
For the nudging term, using the identity
\[
\Delta(K_h*f)=(\Delta K_h)*f,
\]
we obtain
\[
\Delta\Bigl(K_h*(K_h*\nu_t-\mu_t^{\mathrm{obs}})\Bigr)
=(\Delta K_h)*(K_h*\nu_t-\mu_t^{\mathrm{obs}}).
\]
Since $K_h$ is Gaussian, Lemma~\ref{lemma:kernel} implies that $K_h$
is smooth, and in particular $\Delta K_h\in L^1(\Omega)$. By Young's
inequality,
\[
\Bigl\|\Delta\Bigl(K_h*(K_h*\nu_t-\mu_t^{\mathrm{obs}})\Bigr)\Bigr\|_{L^\infty(\Omega)}
\le
\|\Delta K_h\|_{L^1(\Omega)}
\|K_h*\nu_t-\mu_t^{\mathrm{obs}}\|_{L^\infty(\Omega)}.
\]
Therefore,
\[
\|\nabla\cdot c[\nu](\cdot,t)\|_{L^\infty(\Omega)}
\le C,
\qquad \forall t\in[0,T].
\]
We now rewrite the equation as
\[
\partial_t \nu-\nabla\cdot(A\nabla\nu)
+c[\nu]\cdot\nabla\nu+(\nabla\cdot c[\nu])\nu=0.
\]
Define $m(t):=\underline\nu_0 e^{-Ct}$.
Since $m$ is independent of $x$, we have $\nabla m=0$ and
$\nabla\cdot(A\nabla m)=0$, and thus
\[
\partial_t m-\nabla\cdot(A\nabla m)+c[\nu]\cdot\nabla m+(\nabla\cdot c[\nu])m
= m'(t)+(\nabla\cdot c[\nu])m
\le (-C+C)m=0.
\]
Hence $m$ is a subsolution of the same linear parabolic equation satisfied by
$\nu$. Since
\[
\nu(x,0)=\nu_0(x)\ge \underline\nu_0=m(0),
\qquad \forall x\in\Omega,
\]
the parabolic comparison principle yields
\[
\nu(x,t)\ge m(t)=\underline\nu_0 e^{-Ct},
\qquad \forall (x,t)\in\Omega\times[0,T].
\]
The conclusion follows.
\end{proof}

\section{Additional results for the one-dimensional linear benchmark}

We include two additional linear benchmark cases to test the behavior of the nudging correction under stronger mean-reversion bias. In both experiments, the reference dynamics use the same coefficient as in the main text, $a_{\rm true}=1$, while the biased forecast uses $a=2$ or $a=5$. Since these values are larger than the reference value, the forecast model pulls particles too strongly toward the empirical mean and therefore produces an overly concentrated distribution with a smaller variance than the true system.

\begin{figure}[htbp]
    \centering
    \includegraphics[width=0.65\textwidth]{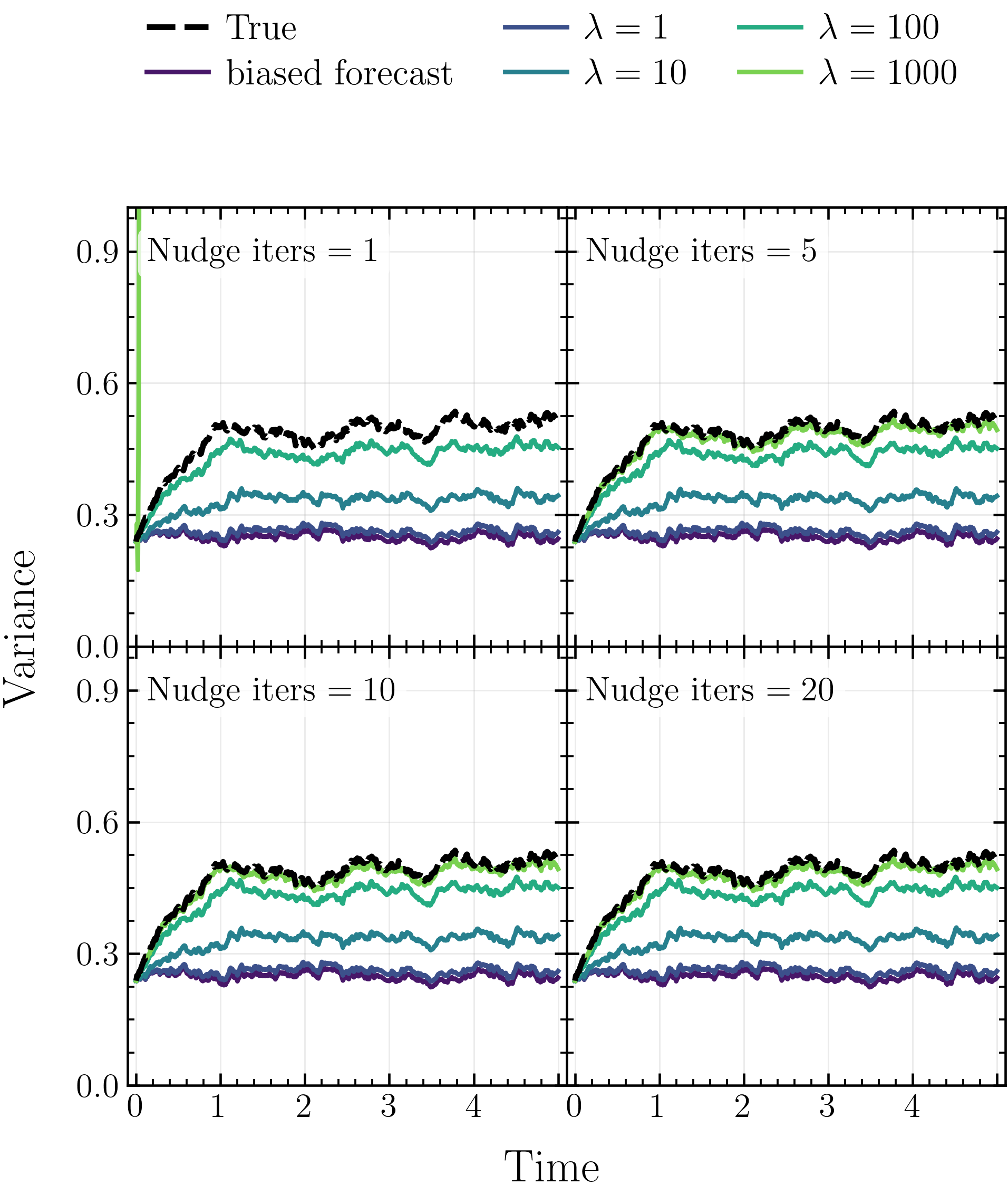}
    \caption{\textbf{Variance dynamics in the one-dimensional linear benchmark ($a=2$).} We compare the reference system, the biased forecast model, and assimilated (nudged) trajectories with $\lambda\in\{1,10,100,1000\}$. Increasing $\lambda$ generally improves tracking accuracy, but excessively large nudging may reduce numerical stability.}
    \label{fig:simple_linear_case_a_2}
\end{figure}

Figure~\ref{fig:simple_linear_case_a_2} shows the moderately over-interacting case $a=2$. The biased forecast remains below the reference variance, reflecting excessive contraction around the mean. The nudging correction increases the variance toward the reference curve, with stronger corrections giving better agreement. The improvement is already visible for intermediate values of $\lambda$, while the largest value can produce a short initial overshoot when the correction is applied with too few nudging substeps.

\begin{figure}[htbp]
    \centering
    \includegraphics[width=0.65\textwidth]{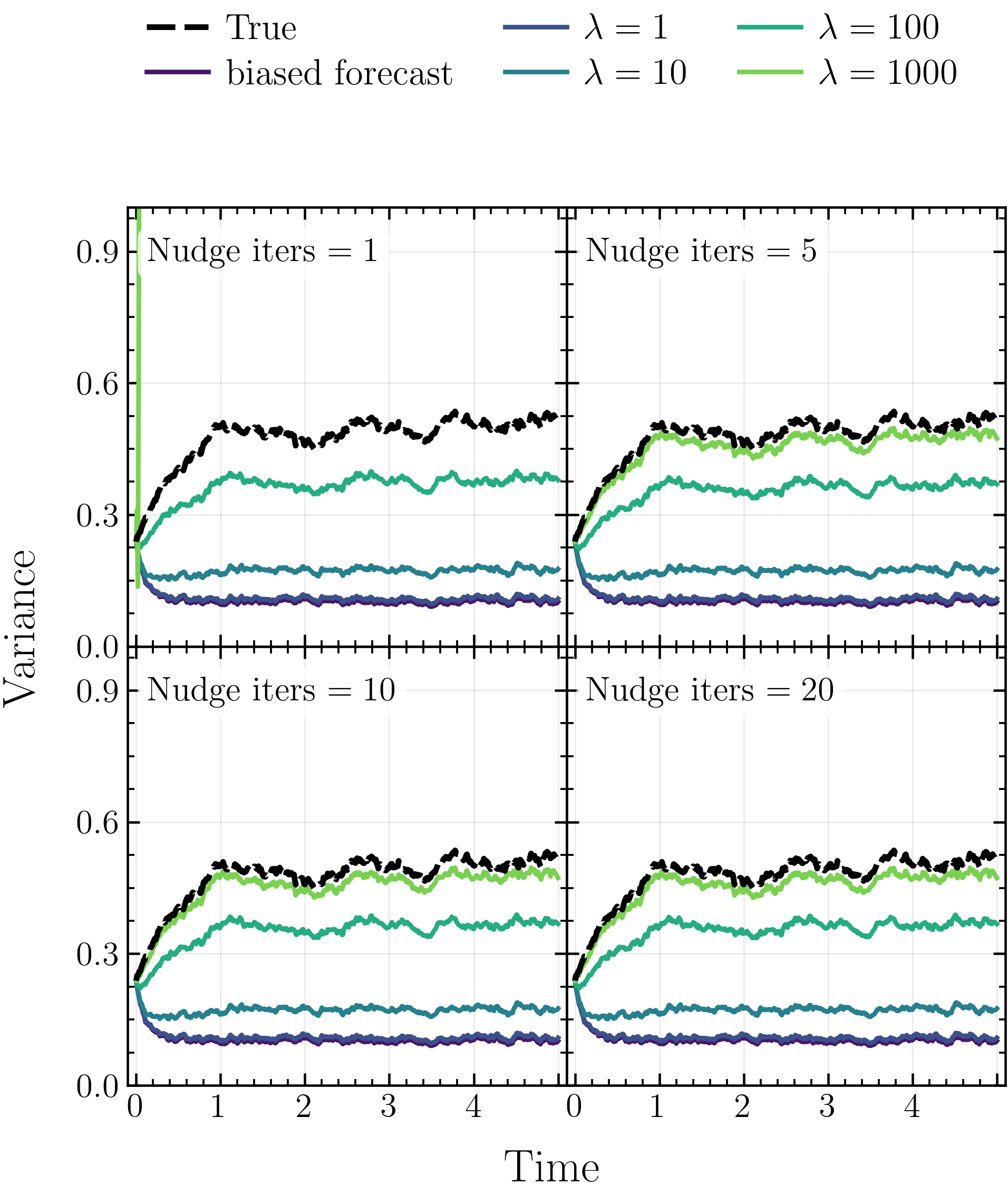}
    \caption{\textbf{Variance dynamics in the one-dimensional linear benchmark ($a=5$).} We compare the reference system, the biased forecast model, and assimilated (nudged) trajectories with $\lambda\in\{1,10,100,1000\}$. Increasing $\lambda$ generally improves tracking accuracy, but excessively large nudging may reduce numerical stability.}
    \label{fig:simple_linear_case_a_5}
\end{figure}

Figure~\ref{fig:simple_linear_case_a_5} considers a more strongly biased forecast model. In this case, the unassimilated trajectory is substantially over-concentrated, so a larger nudging strength is needed to recover the correct variance level. The qualitative trend is consistent with the $a=2$ case; increasing $\lambda$ reduces the variance mismatch, but aggressive nudging can introduce temporary numerical instability. These additional experiments support the robustness of the proposed correction mechanism for both moderate and severe over-interaction bias.

\section{Additional results for the multimodal benchmark}

We include additional experiments for the multimodal benchmark with biased interaction coefficients
$a=0.1$ and $a=0.5$. These two cases complement the main experiment with $a=1.5$ and test whether the proposed nudging correction remains effective when the forecast model is only moderately misspecified. The reference system uses $a_{\rm true}=0.25$, so the case $a=0.1$ corresponds to a weaker interaction than the reference dynamics, while $a=0.5$ corresponds to a stronger interaction.

\begin{figure}[htbp]
    \centering
    \includegraphics[width=0.65\textwidth]{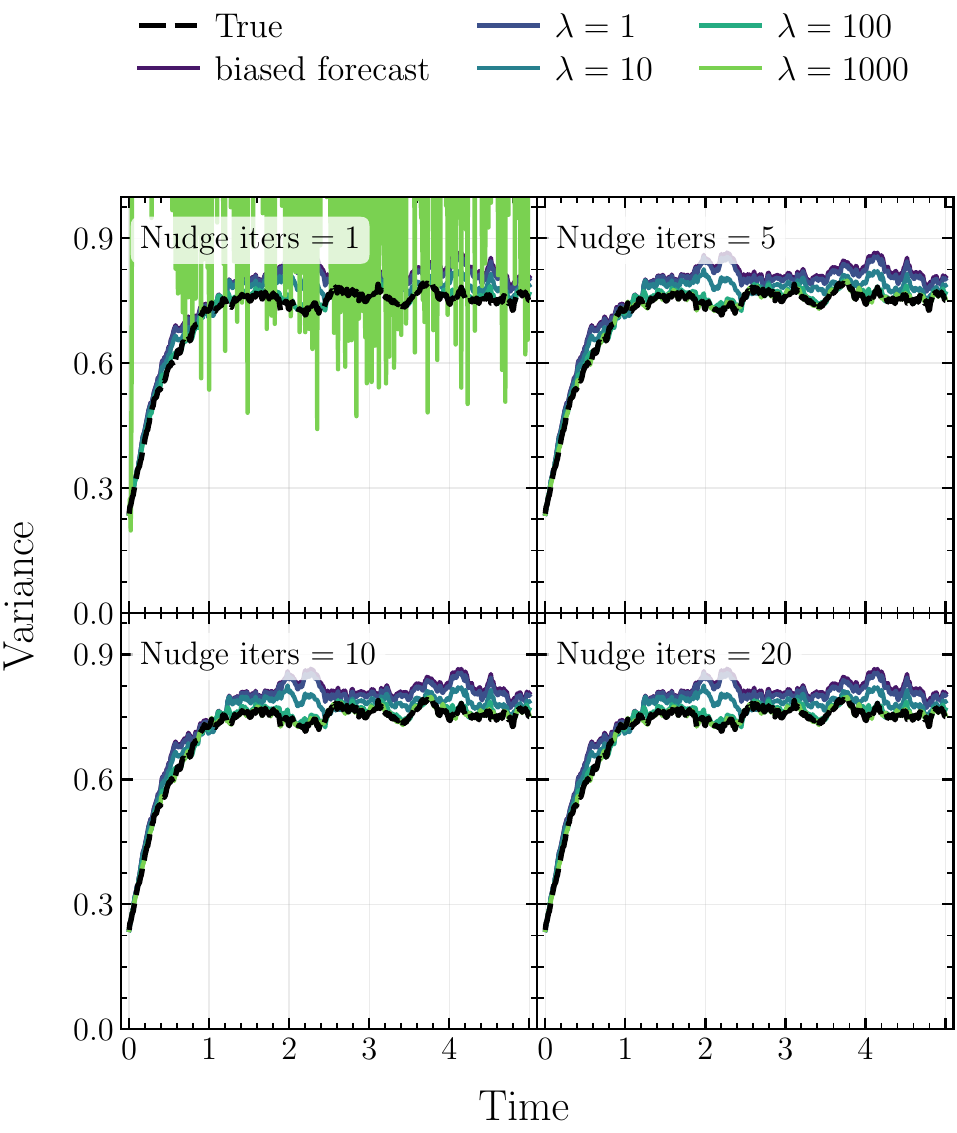}
    \caption{\textbf{Variance dynamics in the multimodal benchmark (\(a=0.1\)).} We compare the reference system, the biased forecast model, and assimilated (nudged) trajectories with \(\lambda\in\{1,10,100,1000\}\). The four panels correspond to different numbers of nudging iterations. Increasing \(\lambda\) improves tracking accuracy, while excessively large nudging can introduce temporary numerical instability when the correction is applied too aggressively.}
    \label{fig:case2_true_0.25_a_0.1}
\end{figure}

Figure~\ref{fig:case2_true_0.25_a_0.1} shows the case where the forecast interaction strength is smaller than the reference value. Since this bias is relatively mild, the unassimilated forecast already follows the reference variance more closely than in the more strongly biased case shown in the main text. Nevertheless, the nudging correction still yields better agreement with the reference trajectory. For moderate and large values of \(\lambda\), the assimilated variance remains close to the true variance after a short adjustment period. As in the previous experiments, using only one nudging iteration together with a large \(\lambda\) can create visible oscillations, reflecting the stability limitation of an explicit correction step.

\begin{figure}[htbp]
    \centering
    \includegraphics[width=0.65\textwidth]{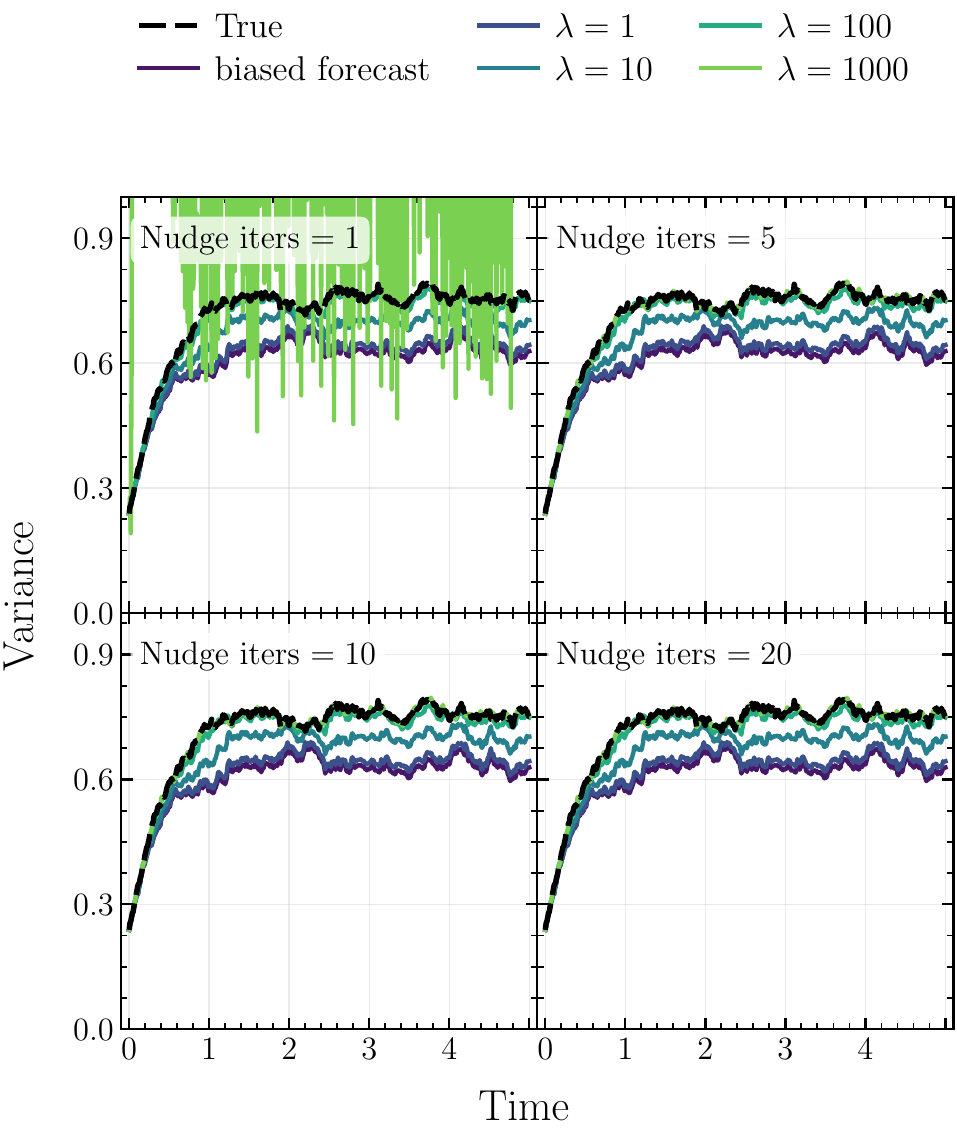}
    \caption{\textbf{Variance dynamics in the multimodal benchmark (\(a=0.5\)).} We compare the reference system, the biased forecast model, and assimilated (nudged) trajectories with \(\lambda\in\{1,10,100,1000\}\). The four panels correspond to different numbers of nudging iterations. Increasing \(\lambda\) improves tracking accuracy, while excessively large nudging can introduce temporary numerical instability when the correction is applied too aggressively.}
    \label{fig:case2_true_0.25_a_0.5}
\end{figure}

Figure~\ref{fig:case2_true_0.25_a_0.5} considers a forecast model with stronger interaction than the reference system. In this case, the biased model tends to contract the distribution too strongly toward the empirical mean, which leads to a variance mismatch. The nudging term corrects this error by pushing the forecast law toward the observed coarse density. Increasing \(\lambda\) generally reduces the variance gap, while the largest value again requires enough nudging substeps to avoid temporary numerical instability.

\begin{figure}[h]
    \centering
    \includegraphics[width=0.65\textwidth]{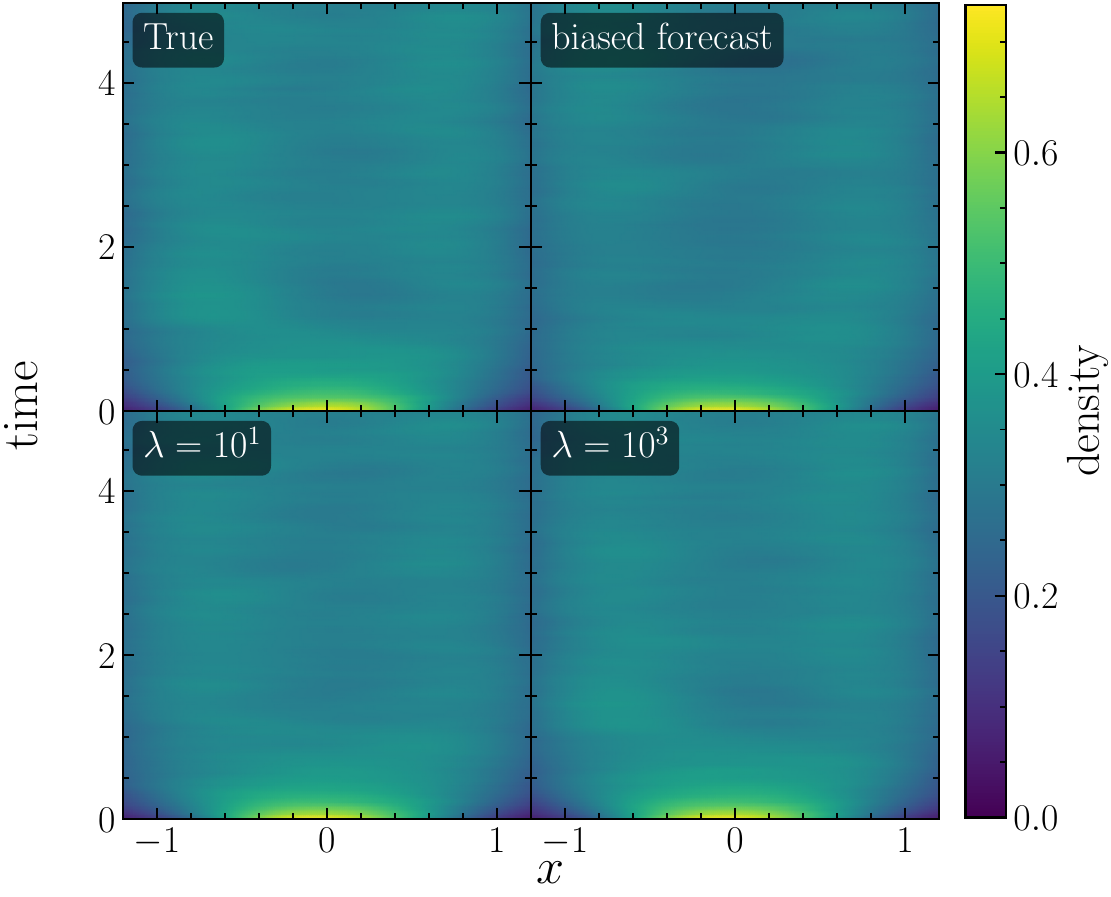}
    \caption{\textbf{Space-time density evolution in the multimodal benchmark (\(a=0.1\)).} Top-left: reference density. Top-right: biased forecast. Bottom-left: assimilated density with \(\lambda=10\). Bottom-right: assimilated density with \(\lambda=1000\). Larger nudging strength restores both the location and the spread of the true law.}
    \label{fig:case2_density_true_0.25_a_0.1}
\end{figure}

The corresponding space-time density in Figure~\ref{fig:case2_density_true_0.25_a_0.1} confirms that the improvement is not limited to the variance. The biased forecast captures the overall bimodal structure but exhibits visible discrepancies in the spread and relative concentration of the density. The assimilated solutions reduce these discrepancies, with the stronger correction giving a density evolution closer to the reference law. This shows that the measure-based feedback can correct the distribution at the level of the full density, not only at the level of a low-order statistic.

\begin{figure}[h]
    \centering
    \includegraphics[width=0.65\textwidth]{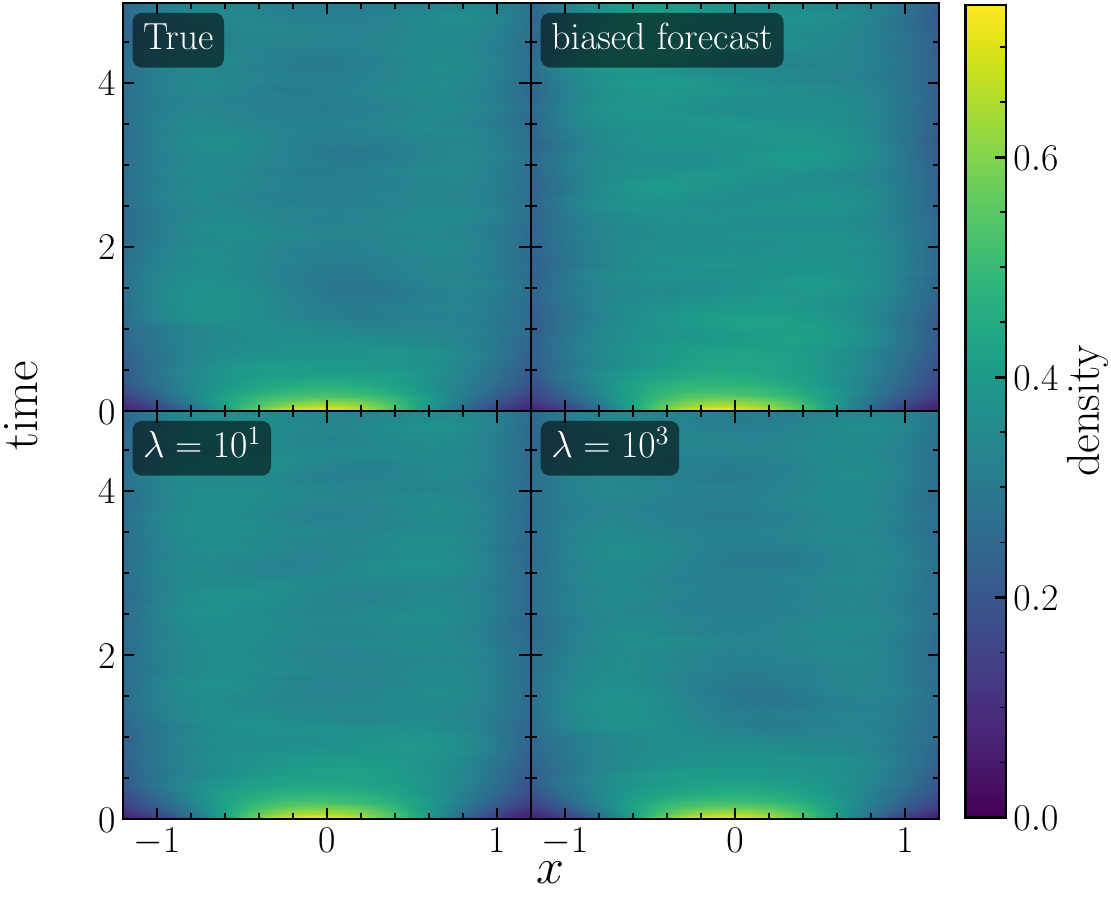}
    \caption{\textbf{Space-time density evolution in the multimodal benchmark (\(a=0.5\)).} Top-left: reference density. Top-right: biased forecast. Bottom-left: assimilated density with \(\lambda=10\). Bottom-right: assimilated density with \(\lambda=1000\). Larger nudging strength restores both the location and the spread of the true law.}
    \label{fig:case2_density_true_0.25_a_0.5}
\end{figure}

Figure~\ref{fig:case2_density_true_0.25_a_0.5} shows a similar trend for the stronger-interaction forecast model. The biased prediction has an incorrect density profile, because the interaction term disrupts the balance between concentration near the center and spreading toward the two wells. Nudging with \(\lambda=10\) partially corrects this mismatch, while \(\lambda=1000\) gives a closer reconstruction of the reference space-time density. Together, Figures~\ref{fig:case2_true_0.25_a_0.1}--\ref{fig:case2_density_true_0.25_a_0.5} indicate that the proposed correction remains effective for both weaker and stronger interaction bias in the multimodal setting.

\end{document}